\documentclass[12pt]{article}
\usepackage[utf8]{inputenc}
\usepackage[width=16.00cm, height=24.00cm]{geometry}
\usepackage{color}
\usepackage{fourier}
\usepackage{mathrsfs}
\usepackage{amsfonts, amssymb, amsmath, amsthm}

\usepackage{hyperref}

\title{Hardy space of translated Dirichlet series}
\author{Tom\'as Fern\'andez Vidal \and Daniel Galicer \and Mart\'in Mereb \and Pablo Sevilla-Peris}
\date{}

\DeclareMathOperator{\id}{id}

\DeclareMathOperator{\re}{Re}

\newtheorem{teo}{Theorem}[section]
\newtheorem{lem}[teo]{Lemma}
\newtheorem{prop}[teo]{Proposition}
\newtheorem{cor}[teo]{Corollary}

\theoremstyle{definition}

\newtheorem{ej}[teo]{Example}

\newtheorem{obs}[teo]{Remark}

\def\C{\mathbb{C}}

\def\N{\mathbb{N}}
\def\eps{\varepsilon}

\begin{document}
	
\maketitle

\begin{abstract}
We study the Hardy space of translated Dirichlet series $\mathcal{H}_{+}$. It consists on those Dirichlet series $\sum a_n n^{-s}$ such that for some (equivalently, every) $1 \leq p < \infty$, the translation $\sum{a_{n}}n^{-(s+\frac{1}{\sigma})}$ belongs to the Hardy space $\mathcal{H}^{p}$ for every $\sigma>0$.  We prove that this set, endowed with the topology induced by the seminorms $\left\{\Vert\cdot\Vert_{2,k}\right\}_{k\in\mathbb{N}}$ (where $\Vert\sum{a_{n}n^{-s}}\Vert_{2,k}$ is defined as  $\big\Vert\sum{a_n n^{-(s+\frac{1}{k})}} \big\Vert_{\mathcal{H}^{2}}$), is a Fr\'echet space which is Schwartz and non nuclear. Moreover, the Dirichlet monomials $\{n^{-s}\}_{n \in \mathbb N}$ are an unconditional Schauder basis of $\mathcal H_+$. 
We also explore the connection of this new space with spaces of holomorphic functions on infinite-dimensional spaces.\\
In the spirit of Gordon and Hedenmalm's work, we completely characterize the composition operator on the Hardy space of translated Dirichlet series. Moreover, we study the superposition operators on $\mathcal{H}_{+}$ and show that every polynomial defines an operator of this kind. We present certain sufficient conditions on the coefficients of an entire function to define a superposition operator. Relying on  number theory techniques we exhibit some examples which do not provide superposition operators. We finally look at the action of the differentiation and integration operators on these spaces. 
\end{abstract}

\footnotetext[0]{\textit{Keywords:} Dirichlet series, Hardy space, Fr\'echet space, composition operator, superposition operator\\
\textit{2010 Mathematics subject classification:} Primary: 46A04 Secondary: 30B50, 32A10, 46A11, 46A45, 46G20, 47B33}
	
\section{Introduction}
	
A Dirichlet series is a formal series of the form $D=\sum a_{n} n^{-s}$, where the coefficients $a_{n}$ are complex numbers and $s$ is a complex variable. Introduced by Dirichlet in the second half of the 19th century, they played a fundamental r\^ole in the development 
of the analytic theory of numbers. In the beginning of the 20th century they started to be studied from the point of view of complex analysis. It was then proved 
(see e.g. \cite{queffelec2013diophantine} or \cite{defant2018Dirichlet}) that Dirichlet series converge on half planes and that, then, for each Dirichlet series $D$ there 
exists $\sigma_{c}(D) \in \mathbb{R}\cup\{\pm \infty\}$ so that the series $\sum_{n=1}^{\infty} a_{n} \frac{1}{n^{s}}$ converges for every $\re s > \sigma_{c}(D)$ and diverges for every $\re s < \sigma_{c}(D)$. For a given $\sigma \in \mathbb{R}$ we denote
\[
	\mathbb{C}_{\sigma} = \{ s \in \mathbb{C} \colon \re s > \sigma  \} \,.
\]
We will write $\mathbb{C}_{+}$ instead of $\mathbb{C}_{0}$. With this notation $\mathbb{C}_{\sigma_{c}(D)}$ is the maximal half plane of convergence of $D$, and $D(s)$ defines a holomorphic function on  $\mathbb{C}_{\sigma_{c}(D)}$. With the same idea, there is an abscissa $\sigma_{a}(D)$ that defines the
maximal half plane of absolute convergence of the series, and an abscissa $\sigma_{u}(D)$ so that the Dirichlet series converges uniformly on $\mathbb{C}_{\sigma}$ for every $\sigma > \sigma_{u}(D)$. \\
	
By the end of the 1990s a deep relation between Dirichlet series and different parts of analysis (mainly harmonic and functional analysis) was discovered. Ever since, this interaction has shown to be very fruitful, with interesting results in both sides. A key element in this development are the Hardy 
spaces $\mathcal{H}^{p}$ of Dirichlet series, introduced by Hedenmalm, Lindqvist and Seip for $p=2$ and $p=\infty$ \cite{hedenmalm1995hilbert} and by Bayart for the remaining cases in $1 \leq  p \leq \infty$ \cite{bayart2002hardy}. Let us 
briefly recall the definition of these spaces. The space $\mathcal{H}^{\infty}$ is defined as consisting of Dirichlet series that define a bounded holomorphic function on $\mathbb{C}_{+}$. Given $1 \leq p < \infty$, the expresion
\[
	\Big\Vert \sum_{n=1}^{N} a_{n} n^{-s} \Big\Vert_{p}
	= \lim_{R \to \infty} \bigg( \frac{1}{2R} \int_{-R}^{R} \Big\vert \sum_{n=1}^{N} a_{n} n^{-it} \Big\vert^{p}  dt  \bigg)^{\frac{1}{p}} \,.
\]
defines a norm on the space of Dirichlet polynomials (i.e. finite Dirichlet series). Then the space $\mathcal{H}^{p}$ is defined as the completion of the Dirichlet polynomials under this norm. Let us note that $\mathcal{H}^{2}$ consists exactly of those Dirichlet series $\sum a_{n} n^{-s}$ for which $(a_{n})_{n} \in \ell_{2}$.\\
Translation is a useful tool within the theory. Given a Dirichlet series $D=\sum a_{n} n^{-s}$ and $\sigma \in \mathbb{R}$ we define the following Dirichlet series
\[
	D_{\sigma} = \sum \frac{a_{n}}{n^{1/\sigma}} n^{-s} \,.
\]
Note that $\sigma_{c} (D_{\sigma}) = \sigma_{c}(D) - \frac{1}{\sigma}$ and, for $\re s > \sigma_{c}(D) - \frac{1}{\sigma}$ we have
\begin{equation} \label{charpentier}
	D_{\sigma} (s) = \sum_{n=1}^{\infty} a_{n} \frac{1}{n^{s + 1/\sigma}}
	= D \big(s + \tfrac{1}{\sigma} \big) \,.
\end{equation}
A key property in this setting is that a Dirichlet series $D$ belongs to $\mathcal{H}^{p}$ if and only if 
\[
	D_{\sigma} \in \mathcal{H}^{p} \text{ for every } \sigma >0 \text{ and } \sup_{\sigma >0} \Vert D_{\sigma} \Vert_{p}< \infty \,.
\]
A natural question then arises: `what happens if we drop the second condition?' or, to be more precise: `what kind of structure do we find when we consider those Dirichlet series for which all translations belong to $\mathcal{H}^{p}$?'. This 
task was overcome by Bonet for $p=\infty$ in \cite{bonet2018frechet}. There the Fr\'echet space $\mathcal{H}_{+}^{\infty}$ of Dirichlet series $D$ so that $D_{\sigma} \in \mathcal{H}^{\infty}$ for every $\sigma >0$ was defined. Several properties of this space, as well as composition operators defined on these spaces were studied. Our aim in this note is to complete this 
study, extending this to $\mathcal{H}^{p}$  for $1 \leq p < \infty$. We define for each $p$ the corresponding Fr\'echet space with the same idea and show that, in fact, they are all isomorphic. So, we denote this space by  $\mathcal{H}_{+}$ and study in Theorem~\ref{teo: caract de esp} certain properties. In particular, we investigate its relation with spaces of holomorphic functions defined on infinite dimensional Banach spaces.\\
	
Given two open sets $\Omega_{1} , \Omega_{2} \subseteq \mathbb{C}$ and a holomorphic function $\phi : \Omega_{1} \to \Omega_{2}$, we have that $f \circ \phi$ is holomorphic for every holomorphic function $f$. This defines in a natural way 
an operator $C_{\phi}$ (called \textit{composition operator of symbol $\phi$}) given by $f \mapsto f \circ \phi$. These are by now very classical objects that have been studied from several points of view. If $\phi : \mathbb{C}_{\sigma_{1}} \to \mathbb{C}_{\sigma_{2}}$ is holomorphic and $f$ is represented by a Dirichlet series on $\mathbb{C}_{\sigma_{2}}$, then $f \circ \phi$ 
is holomorphic on $\mathbb{C}_{\sigma_{1}}$ but may not be represented as a Dirichlet series. One may then wonder under what circumstances is $f \circ \phi$ represented by a Dirichlet series on $\mathbb{C}_{\sigma_{1}}$. Or, to put it in other terms, if $\mathcal{F}$ is some space of Dirichlet series converging on $\mathbb{C}_{\sigma_{2}}$, to find conditions on $\phi$ so that $C_{\phi}$ defines a composition operator acting on $\mathcal{F}$ and taking values on some other space of Dirichlet series. 	
The first ones to address this question were Gordon and Hedenmalm in \cite{gordon1999composition}, who characterised those $\phi$ for which $C_{\phi}$ is a well defined (and continuous) operator on $\mathcal{H}^{2}$ taking values on $\mathcal{D}$ (see below for the definition) or on $\mathcal{H}^{2}$ and studied some of their properties. This study was later extended by Bayart for $\mathcal{H}^{p}$ with $p\neq 2$ in \cite{bayart2002hardy} and Bonet for $\mathcal{H}^{\infty}_{+}$ in \cite{bonet2018frechet}. In Section~\ref{sect:compop} we perform an analogous study for operators $C_{\phi}$ defined on $\mathcal{H}_{+}$, giving conditions on $\phi$ so that the compostion operator $\mathcal{H}_{+} \to \mathcal{H}_{+}$ is well defined, continuous or bounded. We also study when the composition operator takes values in $\mathcal{H}^{\infty}_{+} $ or $\mathcal{H}^{p}$ for some $p$.\\
Superposition operators $S_{\varphi}$ are defined with a similar idea, but with a different point of view. Given a function $\varphi$, these are defined as $f \mapsto \varphi \circ f$. When two spaces of functions are fixed as domain and range, the questions are the same as before: to find conditions on the symbol $\varphi$ so that the operator $S_{\varphi}$ is well defined and continuous. For operators on $\mathcal{H}^{\infty}$and between $\mathcal{H}^{p}$ and $\mathcal{H}^{q}$these questions were addressed in \cite{bayart2019composition}. It was shown there that  $\varphi$ defines a superposition operator $\mathcal{H}^{\infty} \to \mathcal{H}^{\infty}$ if and only if $\varphi$ is entire, whereas those $\varphi$s leading to superposition operators between $\mathcal{H}^{p} \to \mathcal{H}^{q}$ correspond exclusively to certain polynomials (where the degree depends on the involved space). In Section~\ref{sec:super} we study superposition operators defined on $\mathcal{H}_{+}$.  
We show that if $\varphi$ defines a superposition operator $\mathcal{H}_{+} \to \mathcal{H}_{+}$, then it has to be entire. Moreover, $\mathcal{H}_{+}$ results a Fr\'echet algebra and therefore, in particular, every polynomial defines a superposition operator. On the other hand, opposite to the case of $\mathcal{H}^{\infty}$ (or also $\mathcal{H}^{\infty}_{+}$), we exhibit entire functions that do not define operators of this type.\\
A very classical operator acting on spaces of differentiable functions is the differentiation operator, $f \mapsto f'$. The action of this operator and of its inverse on $\mathcal{H}_{+}^{\infty}$ has been studied in \cite{bonetoperator}. In Section~\ref{sect:dif integ} we perform a similar study within this context.

Before we proceed let us fix some basic notation that will be used all along this note. We write $\mathcal{D}$ for the space of Dirichlet series that converge at some point (and then define a holomorphic function on some half plane). We denote $\mathbb{N}_{0} = \mathbb{N} \cup \{0\}$ and $\mathbb{N}_{0}^{(\mathbb{N})} = \bigcup_{n=0}^{\infty} \mathbb{N}_{0}^{n} \times \{0\}$ is the set of multi-indices of arbitrary length with non-negative integer entries. Then the notation $\mathbb{Z}^{(\mathbb{N})}$ is clear. Given a sequence of complex numbers $z = (z_{n})_{n}$ and $\alpha = (\alpha_{1}, \ldots, \alpha_{N}, 0,0 , \ldots) \in \mathbb{Z}^{(\mathbb{N})}$ we denote $z^{\alpha} = \prod_{n=1}^{N} z_{n}^{\alpha_{n}}$. We write $\mathbb{D} (z, r)$ for the disc in the complex plane centred at $z$ and with radius $r$, and $\mathbb{T}$ for $\partial \mathbb{D}(0,1)$.
Finally, $\mathfrak{p} = (\mathfrak{p}_{j})_{j}$ denotes the sequence of prime numbers; sometimes $\mathfrak{p}$ will denote a particular primer number.
	
\section{Spaces of translated Dirichlet series}

As we already mentioned before, our aim is to consider the spaces of those Dirichlet series whose translations belong to $\mathcal{H}^{p}$. Abscissas are a quite convenient way to formulate this idea. Let us note that the space defined by Bonet in \cite{bonet2018frechet} can be reformulated as $\mathcal{H}_{+}^{\infty} = \{ D \in \mathcal{D} \colon \sigma_{u} (D) \leq 0  \}$. For $1 \leq p < \infty$, we define the $\mathcal{H}^{p}$-abscissa of a Dirichlet series $D = \sum a_{n} n^{-s}$ as
\[
\sigma_{p} (D) = \inf \big\{ \sigma \colon \textstyle \sum \frac{a_{n}}{n^{\sigma}} n^{-s} \in \mathcal{H}^{p}  \big\} \,,
\]
and define the space
\[
\mathcal{H}^{p}_{+} =  \{ D \in \mathcal{D} \colon \sigma_{p} (D) \leq 0  \} \,.
\]
For each $k \in \mathbb{N}$ we consider the seminorm 
\[
\Big\Vert \sum{a_{n}n^{-s}} \Big\Vert_{p,k} = \Big\Vert\sum{\frac{a_n}{n^{\frac{1}{k}}}n^{-s}} \Big\Vert_{\mathcal{H}^{p}} \,.
\]
By \cite[Proposition~11.20]{defant2018Dirichlet} we have
$\Vert \cdot \Vert_{p,k} \leq \Vert \cdot \Vert_{p,k+1}$, and we may endow  $\mathcal{H}^{p}_{+}$ with the Fr\'echet topology defined by the  family of seminorms $\left\{\Vert\cdot\Vert_{p,k}\right\}_{k\in\mathbb{N}}$. 
	
\begin{obs} \label{fok}
If a Dirichlet series $\sum a_{n}n^{-s}$ belongs to $\mathcal{H}^{p}$ for some $1 \leq p < \infty$, then the translated series $\sum \frac{a_{n}}{n^{\sigma}}n^{-s}$ belongs to $\mathcal{H}^{q}$ for all $\sigma>0$ and every $p<q<\infty$ (see \cite[Section~3]{bayart2002hardy} or \cite[Theorem~12.9]{defant2018Dirichlet}). 
\end{obs}

The previous observation implies that $\sigma_{p}(D)= \sigma_{q}(D)$ for every $1 \leq p,q < \infty$ and, in particular, all $\mathcal{H}^{p}_{+}$ for $1 \leq p < \infty$ are equal as vector spaces. 
Since the inclusion $\mathcal{H}^{q} \hookrightarrow \mathcal{H}^{p}$ is continuous, the inclusion  $\mathcal{H}^{q}_{+}\hookrightarrow\mathcal{H}^{p}_{+}$  between the Fr\'echet spaces is also continuous. Then an application of the Open Mapping Theorem shows that they are all isomorphic. So, we are dealing with only one Fr\'echet space (that we will denote $\mathcal{H}_{+}$) whose topology can be defined by different families of seminorms $\left\{\Vert\cdot\Vert_{p,k}\right\}_{k\in\mathbb{N}}$. In the following proposition we show how these seminorms relate with each other. Before we proceed, let us briefly recall some basic facts that we are going to use. First of all, we consider $\mathbb{T}^{\infty}$, the infinite product of $\mathbb{T}$ on which we consider the product of the normalised Lebesgue measures. Given a function $f \in L_{1} (\mathbb{T}^{\infty})$ and $\alpha \in \mathbb{Z}^{(\mathbb{N})}$ we consider the Fourier coefficient
\[
	\hat{f}(\alpha) = \int_{\mathbb{T}^{\infty}} f(z) z^{-\alpha} dz \,.
\]
It is a well known fact (see e.g. \cite[Chapter~11]{defant2018Dirichlet}) that for every $1 \leq p \leq \infty$ the Hardy space
\[
	H_{p} (\mathbb{T}^{\infty}) = \{ f \in  L_{p} (\mathbb{T}^{\infty}) \colon \hat{f} (\alpha) = 0 \text{ for every } \alpha \in \mathbb{Z}^{(\mathbb{N})} \setminus \mathbb{N}_{0}^{(\mathbb{N})} \}
\]
is isometrically isomorphic to $\mathcal{H}^{p}$.

\begin{prop}\label{obs: igualdal de espacios}
For every  $1\leq p \leq q <\infty$ and $k \in \mathbb{N}$ there exists $C_{k,p,q}>1$ so that 
\begin{equation} \label{seminormas}
	\Big\Vert  \sum a_{n} n^{-s} \Big\Vert_{p,k} 
	\leq \Big\Vert \sum a_{n} n^{-s} \Big\Vert_{q,k} 
	\leq C_{k,p,q} \Big\Vert \sum a_{n} n^{-s} \Big\Vert_{p,2k} \,,
\end{equation}
for every Dirichlet series in $\mathcal{H}_{+}$.
\end{prop}
\begin{proof}
The first inequality in \eqref{seminormas} is obvious, so it is only left to show that the second inequality holds.
The proof relies heavily on \cite[Section~3]{bayart2002hardy} and \cite[Proposition~2.4 and Theorem~2.5]{depese_2019} (see also \cite[Theorem~12.9 and Proposition~12.10]{defant2018Dirichlet}). First of all, fix $1\leq p \leq q <\infty$, $k \in \mathbb{N}$ and find $j_{0}=j_{0}(p,q,k)\in\mathbb{N}_{0}$ such that $\mathfrak{p}_{j}^{\frac{-1}{2k}}<\sqrt{\frac{p}{q}}$ for every $j>j_{0}$. Then, by \cite[Proposition~2.4]{depese_2019} there exists a unique operator
\[
	T \colon H_{p} (\mathbb{T}^{\infty} ) \rightarrow H_{q} (\mathbb{T}^{\infty})
\]
satisfying $T \big( \sum_{\alpha \in \Lambda} a_{\alpha} z^{\alpha} \big) = \sum_{\alpha \in \Lambda} a_{\alpha} \big(\mathfrak{p}^{-1/(2k)} z \big)^{\alpha}$ for every finite $\Lambda \subseteq \mathbb{N}_{0}^{(\mathbb{N})}$ and $\Vert T \Vert \leq \prod_{j=1}^{j_{0}} \frac{1}{1-\mathfrak{p}_{j}^{\frac{-1}{2k}}}$. Following \cite[Theorem~2.5]{depese_2019}, by doing $M \Big( \sum a_{n} n^{-s} \Big) = \sum a_{n} n^{- \frac{1}{2k}} n^{-s}$ we define an operator $M \colon \mathcal{H}^{p}\rightarrow\mathcal{H}^{q}$ satisfying $\Vert M \Vert = \Vert T \Vert$. Then 
\[
		\Vert D\Vert_{q,k}
		=\Big\Vert \sum \frac{a_{n}}{n^{\frac{1}{k}}}n^{-s} \Vert_{q}
		= \Vert M (D_{2k})\Vert_q
		\leq \Vert M\Vert \; \Vert D_{2k}\Vert_p
		\leq \Big( \prod_{j=1}^{j_{0}} \frac{1}{1-\mathfrak{p}_{j}^{\frac{-1}{2k}}} \Big) \;\Vert D\Vert_{p,2k} \,.
\]
This gives \eqref{seminormas} and completes the proof.
	\end{proof}

\begin{obs} \label{inclusiones1}
A straightforward computation shows that 
\[
	\mathcal{H}^{\infty} \subseteq \mathcal{H}^{p} \subseteq \mathcal{H}^{1} \subseteq
	\mathcal{H}_{+}
\]
for every $1  \leq p \leq \infty$ and 
\[
	\mathcal{H}^{\infty} \subseteq\mathcal{H}_{+}^{\infty} \subseteq \mathcal{H}_{+} \,.
\]
There is, however, no relationship between the $\mathcal{H}^{p}$ spaces and  $\mathcal{H}_{+}^{\infty}$. On the one hand, we take the series $D=\sum a_{n} n^{-s}$ defined by $a_{n}=1$ if $n=2^j$ for some $j\in\mathbb{N}$ and $a_{n}=0$ otherwise. Note that  $\sigma_{u}(D) \leq \sigma_{a}(D)=0$, so $\sum a_{n} n^{-s} \in\mathcal{H}^{\infty}_{+}$, but clearly $(a_{n})_{n} \not\in \ell_{2}$ and then $\sum a_{n} n^{-s}$ does not belong to $\mathcal{H}^{2}$ (nor to any $\mathcal{H}^{p}$ for $2 \leq p < \infty$). As a matter of fact, $(a_{n})_{n}$ does not belong to $\ell_{r}$ for any $1 \leq r < \infty$; then a straightforward application of the Haussdorff-Young inequalities (see e.g. \cite{carandomarcecasevilla2019}) shows that $\sum a_{n} n^{-s} \not\in \mathcal{H}^{p}$ for every $1<p \leq 2$. The same argument shows that the series  $\sum b_{n} n^ {-s}$ given by $b_{n} = j$ if $n=2^{j}$ for $j=0,1,2, \ldots$ and $0$ otherwise belongs to $\mathcal{H}_{+}^{\infty}$ but not to $\mathcal{H}^{1}$. Summing up,
\[
	\mathcal{H}^{\infty}_{+} \not\subseteq \mathcal{H}^{p}
\]
for every $1\leq{p}<\infty$.\\
On the other hand for a fixed $\frac{1}{2} > \varepsilon >0$ the series $ \sum \frac{1}{n^{\frac{1}{2}+\varepsilon}} n^{-s}$ has abscissa of convergence $\sigma_{c} > \frac{1}{2} - \varepsilon >0$ and then cannot belong to $\mathcal{H}_{+}^{\infty}$. However, it clearly belongs to $\mathcal{H}^{2}$ and, since it can be seen as the $\frac{\varepsilon}{2}$-translation of the series $\sum \frac{1}{n^{\frac{1}{2}+\frac{\varepsilon}{2}}} n^{-s}$ (that again belongs to $\mathcal{H}^{2}$), it belongs to $\mathcal{H}^{p}$ for every $1 \leq p < \infty$. Hence
\[
	\mathcal{H}^{p} \not\subseteq \mathcal{H}^{\infty}_{+}
\]
for every $1\leq{p}<\infty$.
\end{obs}
	
\begin{obs}
Let us also observe that if $D \in \mathcal{H}_{+}$, then $D_{1/\varepsilon} \in \mathcal{H}^{2}$ for every $\varepsilon>0$ and, by \cite[Remark~1.8 and Theorem~12.11]{defant2018Dirichlet} $\sigma_{a}(D) = \sigma_{a}(D_{1/\varepsilon}) + \varepsilon \leq \frac{1}{2} + \varepsilon$. This gives $\sigma_{a}(D) \leq \frac{1}{2}$. 
Now, the series $\sum \frac{1}{\sqrt{n}} n^{-s}$ is in $\mathcal{H}_{+}$ and satisfies $\sigma_{a}(D) = \frac{1}{2}$. Then
\begin{equation} \label{sigma-aes}
	\sup_{D \in \mathcal{H}_{+}} \sigma_{a}(D) = \frac{1}{2} \,.
\end{equation}
\end{obs}

Our aim now is to prove the following result, parallel to \cite[Theorem~2.2]{bonet2018frechet}.
	
\begin{teo}\label{teo: caract de esp}
The space $\mathcal{H}_{+}$ is a Fr\'echet-Schwartz,  non-nuclear algebra and the Dirichlet monomials $e_n(s)=n^{-s}$ form an unconditional, non-absolute Schauder basis.
\end{teo}

Before we proceed, let us note that for each fixed $1 \leq p < \infty$ and $k \in \mathbb{N}$ we may consider the following space of Dirichlet series
\[
\mathcal{H}_{k}^{p}:=\left\{\sum{a_{n}n^{-s}}:\sum{\frac{a_n}{n^{\frac{1}{k}}}n^{-s}}\in\mathcal{H}^{p}\right\} \,,
\]
that, with the norm $\Vert  \cdot \Vert_{p,k}$, is a Banach space. Note that  $\mathcal{H}^{p}_{k+1}\subseteq\mathcal{H}^{p}_{k}$, and the inclusion is continuous. By Remark~\ref{fok} (and the comment after it), $\mathcal{H}_{+}$ is the projective limit of the spaces  $\mathcal{H}_{k}^{p}$. Note that for every $p$ we get the same projective limit, so we have
\[
\mathcal{H}_{+}:=\bigcap\limits_{k=1}^{\infty}{\mathcal{H}_{k}^{2}} \,,
\]
endowed with the projective limit topology.\\

The space $\mathcal{H}_{+}$ is Schwartz if the inclusions
\begin{equation} \label{inclusiones}
\id_{k} \colon \mathcal{H}^{2}_{k+1}\longrightarrow\mathcal{H}^{2}_{k}
\end{equation}
(that, as we already mentioned,  are continuous)
are all compact. This in the case of $\mathcal{H}_{+}^{\infty}$ is done in \cite{bonet2018frechet} by using a variant of Montel's theorem for Dirichlet series due to Bayart \cite[Lemma~18]{bayart2002hardy}. In our case, due to the particular structure of the spaces, is particularly easy. Let us note that the mappings $\mathcal{H}^{2}_{k+1}\rightarrow\mathcal{H}^{2}$ given by $\sum a_{n} n^{-s} \mapsto \sum \frac{a_{n}}{n^{1/(k+1)}} n^{-s}$ and $\mathcal{H}^{2}\rightarrow\mathcal{H}^{2}_{k}$ given by $\sum a_{n} n^{-s} \mapsto \sum a_{n}n^{1/k} n^{-s}$ are continuous. Also, the operator $\mathcal{H}^{2} \to \mathcal{H}^{2}$ defined as $\sum a_{n} n^{-s} \mapsto \sum a_{n} \frac{n^{1/(k+1)}}{n^{1/k}} n^{-s}$ is compact (because it is a diagonal operator between Hilbert spaces with a defining sequence tending to $0$). The inclusion $\mathcal{H}^{2}_{k+1}\hookrightarrow\mathcal{H}^{2}_{k}$ is the composition of these three mappings and is, therefore compact. \\
	
This already gives the first statement in Theorem~\ref{teo: caract de esp}, namely that $\mathcal{H}_{+}$ is a Fr\'echet-Schwartz  space (hence Montel and reflexive, see \cite[Remark~24.24]{meise1997introduction}).  The rest of the statements are scattered along the paper: the fact that it is an algebra is proved in Proposition~\ref{Fr alg}, the monomials are shown to form an unconditional Schauder basis in Lemma~\ref{schauder}, and that it is not absolute follows from the identification with certain K\"othe echelon space (see Remark~\ref{koethe}). The non-nuclearity is given in Lemma~\ref{nuclear}.\\

Let us recall that a sequence $\{e_{n}\}_{n}$ in a locally convex space $E$ is a Schauder basis if for every $x \in E$ there is a unique sequence $(x_{n})_{n}$ of scalars	so that $x = \sum_{n=1}^{\infty} x_{n} e_{n}$. 	We now move to the proof of the fact that the monomials $\{n^{-s}\}$ form a Schauder basis of $\mathcal{H}_{+}$. This is already known for the Hardy spaces $\mathcal{H}^{p}$ for $1<p<\infty$ \cite{alemanolsensaksman2014} and for $\mathcal{H}_{+}^{\infty}$ \cite{bonet2018frechet}. In this case we even have that the basis is uncoditional (that is, the series $\sum_{n=1}^{\infty} x_{\pi(n)} e_{\pi(n)}$ converges for every permutation $\pi$ of the natural numbers.

\begin{lem}\label{schauder}
The Dirichlet monomials $\{n^{-s}\}$ form an unconditional Schauder basis of $\mathcal{H}_{+}$.
\end{lem}
\begin{proof}
Take $\sum a_{n}n^{-s} \in \mathcal{H}_{+} $ and fix $k \in \mathbb{N}$. Given $N \in \mathbb{N}$ and a permutation $\pi$, let us denote $F=\{\pi (1) , \ldots , \pi(N) \}$. Then
\[
		\bigg\Vert \sum a_{n} n^{-s} - \sum_{n=1}^{N} a_{\pi(n)} \pi(n)^{-s} \bigg\Vert_{2,k}
		= \bigg\Vert \sum_{n \in \mathbb{N} \setminus F} \frac{a_n}{n^{\frac{1}{k}}} n^{-s} \bigg\Vert_{\mathcal{H}^{2}}
		= \bigg(   \sum_{n \in \mathbb{N} \setminus F} \frac{\vert a_n \vert^{2}}{n^{2/k}} \bigg)^{\frac{1}{2}} \,.
\]
But the sequence $\big(\frac{\vert a_{n} \vert^{2}}{n^{2/k}}\big)_{n}$ is absolutely summable (since  $\sum \frac{a_{n}}{n^{1/k}}n^{-s} \in \mathcal{H}^{2} $), hence it converges unconditionally. This gives the conclusion.
\end{proof}
	
Let us recall that a locally convex space $E$ is nuclear if for every seminorm $p$ there exist a seminorm $q$ such that the identity operator $I:(E,q)\rightarrow(E,p)$ is nuclear. Every nuclear space is Schwartz and has a fundamental system of Hilbert seminorms (see \cite[Corollary~28.5 and Lemma~28.1]{meise1997introduction}). The space $\mathcal{H}_{+}$ shares these two properties (note that the family $\big(\Vert \cdot \Vert_{2,k}\big)_{k}$ is a  fundamental system of Hilbert seminorms). It comes then naturally to ask whether or not it is nuclear. 

\begin{lem} \label{nuclear}
The space $\mathcal{H}_{+}$ is not nuclear.
\end{lem}
\begin{proof}
Once we have that the monomials form a Schauder basis, this follows from a straighforward application of the Grothendieck-Pietsch criterion (see e.g. \cite[Theorem~28.15]{meise1997introduction}).
\end{proof}
	
\begin{obs} \label{koethe}
A K\"othe matrix is a sequence $B=(b_{k})_{k\in\mathbb{N}}$, of functions $b_{k} : I \to \mathbb{R}$ (where $I$ is a countable set of indices) satisfying that $0\leq{b_{k}(i)}\leq{b_{k+1}(i)}$, for all $k\in\mathbb{N}$ and all 
$i\in{I}$ and that  for each $i\in{I}$ there exist $k\in\mathbb{N}$ such that $b_{k}(i)>0$. Given a K\"othe matrix and $1 \leq p < \infty$, the corresponding K\"othe echelon space is defined as
\[
	\lambda_{p}(B):=\bigg\{ x\in\mathbb{C}^{I} \colon q_{k}^{(p)}(x):=\bigg(\sum_{i\in{I}} |b_{k}(i)x_i|^p \bigg)^{\frac{1}{p}} < \infty, \text{ for all } k \in\mathbb{N}\bigg\} \,.
\]
These are all Fr\'echet spaces endowed with the topology given by the increasing sequence of seminorms $q_{1}^{(p)}\leq{q_{2}^{(p)}}\leq\cdots\leq{q_{k}^{(p)}}\leq\cdots$. Observe that taking $I=\mathbb{N}$ and defining the matrix $B$ as
\begin{equation} \label{bes}
b_{k} (n) = \frac{1}{n^{\frac{1}{k}}} 
\end{equation}
for $k,n \in \mathbb{N}$ a straightforward computation shows that 
\begin{equation} \label{rial}
\mathcal{H}_{+} = \lambda_{2}(B)
\end{equation}
as Fr\'echet spaces. \\ 
Let us recall that a Shauder basis $\{e_{n}\}_{n}$ of a locally convex space $E$ is absolute if for every continuous seminorm $p$ on $E$ there is a continuous seminorm $q$ on $E$ and $C>0$ so that 
\[
\sum_{n} \vert x_{n} \vert p(e_{n}) \leq C q(x)
\]
for every $x \in E$. Let us note that $\Vert n^{-s} \Vert_{2,k} = \frac{1}{n^{1/k}}$ for every $n,k$. If the monomials were an absolute basis of $\mathcal{H}_{+}$ then, by \cite[Lemma~27.25]{meise1997introduction}, $\mathcal{H}_{+}$ would be isomorphic to $\lambda_{1}(B)$. But \cite[Proposition~28.16]{meise1997introduction} shows that this is not possible. Hence the monomials cannot build an absolute Schauder basis.
\end{obs}

\begin{obs}
Recall from Remark~\ref{inclusiones1} that $\mathcal{H}^{\infty}_{+}  \subseteq \mathcal{H}_{+}$ (as sets), and by the definition of the seminorms, the inclusion  is continuous . Furthermore, 
being $\mathcal{H}^{\infty}_{+}$   a Fr\'echet-Schwartz space \cite[Theorem~2.2]{bonet2018frechet} it is Montel \cite[Remark~24.24]{meise1997introduction}, and the inclusion $\mathcal{H}^{\infty}_{+} \hookrightarrow \mathcal{H}_{+}$ is Montel. From \cite[Proposition~2.3]{bonet2018frechet} and \eqref{rial}, both spaces are not isomorphic.
\end{obs}

\section{Connection with spaces of holomorphic functions}

It is a well known and important fact within the theory that Dirichlet series are closely related with holomorphic functions on infinite-dimensional Banach spaces. The space $\mathcal{H}_{\infty}$ is isometrically isomorphic to the the space of bounded holomorphic functions on $B_{c_{0}}$ (the open unit ball of $c_{0}$) and each $\mathcal{H}^{p}$ is isometrically isomorphic to a certain Hardy space of holomorphic functions on $\ell_{2} \cap \mathbb{D}^{\mathbb{N}}$ (see \cite{queffelec2013diophantine}, \cite[Chapters~3 and~13]{defant2018Dirichlet} or \cite{badefrmase2017} for details). Our aim in this section is to see to what extent can we connect the new spaces with spaces of holomorphic functions.\\
First of all, let us recall that a function $f: U \to \mathbb{C}$ (where $U$ is an open subset of some normed space $X$) is said to be holomorphic if it is Fr\'echet differentiable at every point of $U$. We say that $X$ is a Banach sequence space if it is a vector subspace of $\mathbb{C}^{I}$ (where $I$ is either a finite set or $\mathbb{N}$) endowed with a complete norm and satisfying that, if $x, y \in \mathbb{C}^{I}$ are so that $x \in X$ and $\vert y_{i} \vert \leq \vert x_{i} \vert$ for every $i$, then $y \in X$ and $\Vert y \Vert \leq \Vert x \Vert$. An open subset $R \subseteq X$ is a complete Reinhardt domain if whenever $x \in R$ and $y \in \mathbb{C}^{I}$ are so that $\vert y_{i} \vert \leq \vert x_{i} \vert$ for every $i$, then $y \in R$. If $R$ is a complete Reinhardt domain, then every holomorphic function $f: R \to \mathbb{C}$ defines a unique family of coefficients $\big( c_{\alpha}(f) \big)_{\alpha}$ (where $\alpha$ runs over $\mathbb{N}_{0}^{I}$ if $I$ is finite and on  $\mathbb{N}_{0}^{(\mathbb{N})}$  if $I$ is $\mathbb{N}$). If $I$ is finite then $f(z) = \sum_{\alpha} c_{\alpha}(f) z^{\alpha}$ for every $z \in R$, while if $I=\mathbb{N}$ this may not be the case. A detailed account on these topics can be found in \cite[Chapter~15]{defant2018Dirichlet}.\\

For $N$ and $k$ we write 
\[
\mathfrak{p}^{-1/k}  \mathbb{D}^{N} = \mathfrak{p}^{-1/k}_{1}  \mathbb{D} \times \cdots \times \mathfrak{p}^{-1/k}_{N}  \mathbb{D}  
= \big\{ z \in \mathbb{C}^{N} \colon \vert z_{j} \vert < \mathfrak{p}_{j}^{-1/k}, \, j = 1, \ldots , N
\big\}
\]
and define the space $H^{p} (\mathfrak{p}^{-1/k}  \mathbb{D}^{N} )$ as the space of holomorphic functions $g: \mathfrak{p}^{-1/k}  \mathbb{D}^{N}  \to \mathbb{C}$ so that
\[
\Vert g \Vert_{p} :=
\sup_{\genfrac{}{}{0pt}{}{0<r_{j}< \mathfrak{p}_{j}^{-1/k}}{ j = 1, \ldots , N}}
\bigg( \int_{\mathbb{T}^{N}} \vert g(r_{1}z_{1}, \ldots , r_{N}z_{N}) \vert^{p} dz 
\bigg)^{\frac{1}{p}} < \infty \,.
\]
To functions of $N$ variables correspond Dirichlet series that depend only on the first $N$ primes. We denote $\mathcal{P}_{N} = \{ \mathfrak{p}_{1}^{\alpha_{1}} \in \mathbb{N} \cdots \mathfrak{p}_{N}^{\alpha_{N}}  \colon (\alpha_{1}, \ldots , \alpha_{N}) \in \mathbb{N}_{0}^{N} \}$ and consider the space
\[
\mathcal{H}_{k}^{p,(N)} = \big\{ \textstyle\sum a_{n} n^{-s} \in \mathcal{H}_{k}^{p} \colon a_{n} \neq 0 \Rightarrow n \in \mathcal{P}_{N} \big\} \,.
\]

\begin{prop} \label{praetroius}
For every $k, N\in \mathbb{N}$ and $1 \leq p < \infty$ we have
$\mathcal{H}_{k}^{p,(N)} = H^{p} (\mathfrak{p}^{-1/k}  \mathbb{D}^{N} )$ and, if $\sum a_{n} n^{-s}$ and $g$ are related to each other, then $a_{n} = c_{\alpha}(g)$ whenever $n = \mathfrak{p}^{\alpha}$. 
\end{prop}
\begin{proof}
Choose some $\sum a_{n} n^{-s} \in \mathcal{H}_{k}^{p,(N)}$, then $\sum \frac{a_{n}}{n^{1/k}} n^{-s} \in \mathcal{H}^{p}$ and depends only on the first $N$ primes. Then (see for example \cite[page~316]{defant2018Dirichlet}) we can find $f \in H^{p}(\mathbb{D}^{N})$ so that $c_{\alpha}(f)= \frac{a_{n}}{n^{1/k}} $ whenever $n = \mathfrak{p}^{\alpha}$ and 
\[
\Vert f \Vert_{p} = \big\Vert \textstyle\sum \frac{a_{n}}{n^{1/k}} n^{-s} \big\Vert_{\mathcal{H}^{p}}
= \big\Vert \textstyle\sum a_{n} n^{-s} \big\Vert_{\mathcal{H}_{k}^{p}} \,.
\]
Define now a function $g : \mathfrak{p}^{-1/k}  \mathbb{D}^{N}  \to \mathbb{C}$ by $g(z)= f \big( \mathfrak{p}^{1/k}_{1} z_{1}, \ldots  , \mathfrak{p}^{1/k}_{N} z_{N} \big)$. This is clearly holomorphic and
\[
\sup_{\genfrac{}{}{0pt}{}{0<r_{j}< \mathfrak{p}_{j}^{-1/k}}{ j = 1, \ldots , N}}
\bigg( \int_{\mathbb{T}^{N}} \vert g(r_{1}z_{1}, \ldots , r_{N}z_{N}) \vert^{p} dz 
\bigg)^{\frac{1}{p}} 
= \sup_{\genfrac{}{}{0pt}{}{0<s_{j}< 1}{ j = 1, \ldots , N}}
\bigg( \int_{\mathbb{T}^{N}} \vert f(s_{1}z_{1}, \ldots , s_{N}z_{N}) \vert^{p} dz 
\bigg)^{\frac{1}{p}} 
= \Vert f \Vert_{H^{p}(\mathbb{D}^{N})} \,.
\]
Hence $g \in H^{p} (\mathfrak{p}^{-1/k}  \mathbb{D}^{N} )$ and, moreover,
\[
g(z) =  f\big( \mathfrak{p}^{1/k}_{1} z_{1}, \ldots  , \mathfrak{p}^{1/k}_{N} z_{N} \big) 
= \sum_{\alpha \in \mathbb{N}_{0}^{N}} c_{\alpha}(f) \big( \mathfrak{p}^{1/k} z \big)^{\alpha}
= \sum_{\alpha \in \mathbb{N}_{0}^{N}} c_{\alpha}(f) \big( \mathfrak{p}^{1/k} \big)^{\alpha} z^{\alpha}
\]
for every $z \in \mathfrak{p}^{-1/k}  \mathbb{D}^{N}$. By the uniqueness of the monomial coefficients,
\[
c_{\alpha} (g) = c_{\alpha}(f) \big( \mathfrak{p}^{1/k} \big)^{\alpha} = \frac{a_{n}}{n^{1/k}}n^{1/k} = a_{n}
\]
if $n = \mathfrak{p}^{\alpha}$, and $\mathcal{H}_{k}^{p,(N)} \hookrightarrow H^{p} (\mathfrak{p}^{-1/k}  \mathbb{D}^{N} )$.\\
On the other hand, given $g \in H^{p} (\mathfrak{p}^{-1/k}  \mathbb{D}^{N} )$ define $a_{n} = c_{\alpha}(g)$ for $n= \mathfrak{p}^{\alpha}$ and consider the Dirichlet series $\sum a_{n}n^{-s}$ (note that $a_{n}=0$ for every $n \notin \mathcal{P}_{N}$). Essentially the same computations as before show that the function $f: \mathbb{D}^{N} \to \mathbb{C}$ given by $f(z) = g\big( \mathfrak{p}^{-1/k}_{1} z_{1}, \ldots  , \mathfrak{p}^{-1/k}_{N} z_{N} \big) $ belongs to $H^{p}(\mathbb{D}^{N})$ and $c_{\alpha}(f) = \frac{c_{\alpha}(g)}{(\mathfrak{p}^{1/k})^{\alpha}}$. Then (recall again \cite[page~316]{defant2018Dirichlet}) we can find $\sum b_{n} n^{\-s} \in \mathcal{H}^{p}$ so that $b_{n} = c_{\alpha} (f) = \frac{c_{\alpha}(g)}{(\mathfrak{p}^{1/k})^{\alpha}} =  \frac{a_{n}}{n^{1/k}}$. This gives $\sum a_{n} n^{-s} \in \mathcal{H}_{k}^{p,(N)}$ and completes the proof.
\end{proof}

In $H(\mathbb{D}^{N})$, the space of all holomorphic functions on $\mathbb{D}^{N}$ we define, for each $k$ and $p$,
\[
\rho_{k,p} (f) = \bigg( \int_{\mathbb{T}^{N}} \big\vert f \big( \mathfrak{p}_{1}^{-1/k} z_{1} , \ldots ,  \mathfrak{p}_{N}^{-1/k} z_{N} \big) \big\vert^{p} dz \bigg)^{\frac{1}{p}} \,.
\]

\begin{prop} \label{gamelan}
\[
\mathcal{H}_{+}^{(N)} := \big\{ \textstyle\sum a_{n} n^{-s} \in \mathcal{H}_{+} \colon a_{n} \neq 0 \Rightarrow n \in \mathcal{P}_{N} \big\}  =  H(\mathbb{D}^{N}) 
\]
and, if $f$ and $\sum a_{n} n^{-s}$ are associated to each other,
\[
\big\Vert \textstyle\sum a_{n} n^{-s} \big\Vert_{p,k} = \rho_{k,p} (f)
\]
for every $k$ and $1 \leq p < \infty$.
\end{prop}	
\begin{proof}
Fix some $1 \leq p < \infty$ and take $\sum a_{n} n^{-s} \in \mathcal{H}_{+}^{(N)}$. Then, for each $k$, the Dirichlet series belongs to $\mathcal{H}_{k}^{p,(N)}$ and, by Proposition~\ref{praetroius} we can find $g_{k} \in H^{p} (\mathfrak{p}^{-1/k} \mathbb{D}^{N})$ so that $c_{\alpha} (g_{k}) = a_{n}$ for every $n= \mathfrak{p}^{\alpha}$. By the uniqueness of the coefficients one easily gets that $g_{k}|_{\mathfrak{p}^{-1/j} \mathbb{D}^{N}} = g_{j}$ for $k \geq j$ and, therefore, we may define a holomorphic function $f : \mathbb{D}^{N} \to \mathbb{C}$ so that $c_{\alpha}(f) = a_{n}$. Moreover $\rho_{k,p} (f)= \Vert g_{k} \Vert_{H^{p} (\mathfrak{p}^{-1/k} \mathbb{D}^{N})}
= \big\Vert \textstyle \sum a_{n} n^{-s} \big\Vert_{p,k}$. \\
On the other hand, the restriction of every holomorphic function $f : \mathbb{D}^{N} \to \mathbb{C}$ clearly belongs to $H^{p} (\mathfrak{p}^{-1/k} \mathbb{D}^{N})$ (and the norm equals $\rho_{k,p}(f)$). Proposition~\ref{praetroius} gives that $\sum a_{n} n^{-s}$ (where  $a_{n}= c_{\alpha}(f)$) belongs to $\mathcal{H}_{k}^{p,(N)}$ for every $k$ and, hence to $\mathcal{H}_{+}^{(N)}$.
\end{proof}

\begin{prop} \label{Hc}
Let $E$ be either $\mathcal{H}_{+}^{\infty}$ or $\mathcal{H}_{+}$ and $\Vert \cdot \Vert_{k}$ denote in each case either $\Vert \cdot \Vert_{\infty,k}$ or $\Vert \cdot \Vert_{2,k}$. Then $\sum a_{n} n^{-s} \in E$ if and only if
$\sum_{n \in \mathcal{P}_{N}} a_{n} n^{-s} \in E$ for every $N$ and $\sup_{N} \big\Vert \sum_{n \in \mathcal{P}_{N}} a_{n} n^{-s} \big\Vert_{k} < \infty$ for every $k$.
\end{prop}

For each $1 \leq p < \infty$ we define the space $H_{+}^{p} (\ell_{2} \cap \mathbb{D}^{\mathbb{N}})$ as consisting of those holomorphic functions $f : \ell_{2} \cap \mathbb{D}^{\mathbb{N}} \to \mathbb{C}$ satisfying
\begin{equation} \label{jansen}
\sup_{N} 
\bigg( \int_{\mathbb{T}^{N}} \big\vert f \big( \mathfrak{p}_{1}^{-1/k} z_{1} , \ldots ,  \mathfrak{p}_{N}^{-1/k} z_{N}  , 0,0, \ldots \big) \big\vert^{p} dz \bigg)^{\frac{1}{p}} < \infty 
\end{equation}
for every $k$. \\

Given $f \in H_{+}^{p} (\ell_{2} \cap \mathbb{D}^{\mathbb{N}})$ we may consider coefficients $a_{n} = c_{\alpha}(f)$ (with $n=\mathfrak{p}^{\alpha}$) and the Dirichlet series $\iota(f)=\sum a_{n} n^{-s}$. Note that the restriction $f_{N}$ of $f$ to $\mathbb{D}^{N}$ is obviously holomorphic and $\big( \rho_{k,p}(f_{N}) \big)_{N}$ is bounded by the supremum in \eqref{jansen}. Then Propositions~\ref{gamelan} and~\ref{Hc} yield $\sum a_{n} n^{-s} \in \mathcal{H}_{+}$, and therefore the mapping
\begin{equation}\label{inclusion}
 \iota: H_{+}^{p} (\ell_{2} \cap \mathbb{D}^{\mathbb{N}}) \rightarrow \mathcal{H}_{+}\,
\end{equation}
is an inclusion for every $1 \leq p <\infty$. 

It is very natural to wonder if the previous identification is in fact an isomporhism of F\'echet spaces.  The following example shows that this mapping is not surjective. 

\begin{ej}
The series $D= \sum \frac{1}{\sqrt{\mathfrak{p}_n}} \mathfrak{p}_n^{-s}$ is in $\mathcal H_+$ but it does not belong to $\iota( H_{+}^{p} (\ell_{2} \cap \mathbb{D}^{\mathbb{N}}))$.
To see that this is indeed the case, suppose there is a function $f \in H_+^p (\ell_{2} \cap \mathbb{D}^{\mathbb{N}})$ such that $\iota(f)=D$. Then, $c_{\alpha} (f) = \frac{1}{\sqrt{\mathfrak{p}_n}}$ if $\alpha = e_n$ and $c_{\alpha} (f)=0$ otherwise.  \\
Define now the sequence $z=(z_n)$ as $z_1=1/2$, $z_2=1/2$ and $z_n= 1/ ( \sqrt{n \log(n)} \log \log(n))$ for $n\geq 3$. It is easy to see that $z$ lies in the set $\ell_2 \cap \mathbb{D}^{\mathbb{N}}$. 
For each $N \in \mathbb{N}$ we have 
\[
f(z_1, \ldots, z_N, 0, 0, \dots)
= \sum_{\alpha \in \mathbb{N}_{0}^{N}} c_{\alpha} (f) z^{\alpha} 
=\sum_{n=1}^N  \frac{1}{\sqrt{\mathfrak{p}_n}} z_n \,.
\]
On the other hand, the truncates $\big( (z_1, \ldots, z_N, 0, 0, \dots) \big)_{N}$ clearly converge (in $\ell_{2}$) to $z$. 
Thus, continuity and the prime number theorem yield
\[
f(z) = \lim\limits_{N \to \infty} f(z_1, \cdots, z_N, 0,0, \dots) = \lim\limits_{N \to \infty} \sum\limits_{n=1}^N  \frac{1}{\sqrt{\mathfrak{p}_n}} z_n \geq \lim\limits_{N \to \infty} \sum\limits_{n=3}^N \frac{C}{n\log(n) \log\log(n)} =+\infty.
\]
This is a contradiction and proves our claim.
\end{ej}

 We give a characterization of the space $\mathcal H_{+}$ in terms of holomorphic functions in Corollary  \ref{identificacion}. 
To do that we begin by considering the following two weighted sequence spaces
\[
\ell_{\infty} (\mathfrak{p}^{1/k}) = \big\{ z \in \mathbb{C}^{\mathbb{N}} \colon \Vert z \Vert_{\infty, \mathfrak{p}^{1/k}}:= \sup_{n} \vert z_{n} \mathfrak{p}_{n}^{1/k} \vert < \infty  \big\}
\]
and
\[
\ell_{2} (\mathfrak{p}^{1/k}) = \Big\{ z \in \mathbb{C}^{\mathbb{N}} \colon \Vert z \Vert_{2, \mathfrak{p}^{1/k}}:= \Big( \sum_{n} \vert z_{n} \mathfrak{p}_{n}^{1/k} \vert^{2} \Big)^{\frac{1}{2}} < \infty  \Big\} \,.
\]
Both are Banach sequence spaces, and the set $\ell_{2} (\mathfrak{p}^{1/k}) \cap B_{\ell_{\infty} (\mathfrak{p}^{1/k}) }$ is an open, complete Reinhardt domain in $\ell_{2} (\mathfrak{p}^{1/k}) $, and note that $\big( \ell_{2} (\mathfrak{p}^{1/k}) \cap B_{\ell_{\infty} (\mathfrak{p}^{1/k}) }\big) \cap \mathbb{C}^{N} = \mathfrak{p}^{-1/k} \mathbb{D}^{N}$ for every $N$. Now, for each $1 \leq p < \infty$ we define the space $H^{p} (\ell_{2} (\mathfrak{p}^{1/k}) \cap B_{\ell_{\infty} (\mathfrak{p}^{1/k}) })$ as consisting of those holomorphic functions on $\ell_{2} (\mathfrak{p}^{1/k}) \cap B_{\ell_{\infty} (\mathfrak{p}^{1/k}) }$ so that
\[
\Vert f \Vert := \sup_{N} \sup_{\substack{ 0 < r_{j} < \mathfrak{p}_{j}^{-1/k} \\ j = 1, \ldots , N  }}
\bigg( \int_{\mathbb{T}^{N}} \big\vert f \big( r_{1} z_{1} , \ldots ,  r_{N} z_{N}, 0,0, \ldots \big) \big\vert^{p} dz \bigg)^{\frac{1}{p}} < \infty\,.
\]

\begin{prop}
Let $k \in \mathbb{N}$, then
\[
\mathcal{H}^{2}_{k} = H^{2} (\ell_{2} (\mathfrak{p}^{1/k}) \cap B_{\ell_{\infty} (\mathfrak{p}^{1/k}) })
\]
as Banach spaces.
\end{prop}
\begin{proof}
Let us begin by taking some $f \in 	H^{2} (\ell_{2} (\mathfrak{p}^{1/k}) \cap B_{\ell_{\infty} (\mathfrak{p}^{1/k}) })$, defining as usual $a_{n} = c_{\alpha}(f)$ for $n= \mathfrak{p}^{\alpha}$ and considering the Dirichlet series $\sum a_{n}n^{-s}$. Fix $N$ and define $f_{N}$ as the restriction of $f$ to $\mathfrak{p}^{-1/k} \mathbb{D}^{N}$, which is holomorphic and satisfies that 
\begin{multline*}
 \sup_{\substack{ 0 < r_{j} < \mathfrak{p}_{j}^{-1/k} \\ j = 1, \ldots , N  }}
\bigg( \int_{\mathbb{T}^{N}} \big\vert f_{N} \big( r_{1} z_{1} , \ldots ,  r_{N} z_{N} \big) \big\vert^{2} dz \bigg)^{\frac{1}{2}} \\
=  \sup_{\substack{ 0 < r_{j} < \mathfrak{p}_{j}^{-1/k} \\ j = 1, \ldots , N  }}
\bigg( \int_{\mathbb{T}^{N}} \big\vert f \big( r_{1} z_{1} , \ldots ,  r_{N} z_{N}, 0,0, \ldots \big) \big\vert^{2} dz \bigg)^{\frac{1}{2}} 
\leq \Vert f \Vert_{H^{2} (\ell_{2} (\mathfrak{p}^{1/k}) \cap B_{\ell_{\infty} (\mathfrak{p}^{1/k}) })} < \infty \,.
\end{multline*}
Hence $f_{N} \in H^{2} (\mathfrak{p}^{-1/k} \mathbb{D}^{N})$ and, by Proposition~\ref{praetroius}, if $b_{n}= c_{\alpha} (f_{N})$ for every $n=\mathfrak{p}^{\alpha}$ with $\alpha \in \mathbb{N}_{0}^{N}$, then the Dirichlet series $\sum b_{n} n^{-s}$ belongs to $\mathcal{H}^{2,(N)}_{k}$. But, since $c_{\alpha}(f_{N}) = c_{\alpha}(f)$ for every  $\alpha \in \mathbb{N}_{0}^{N}$, this implies that $\sum_{n \in \mathcal{P}_{N}} a_{n} n^{-s} \in \mathcal{H}^{2,(N)}_{k}$ for every $N$ or, in other words, $\sum_{n \in \mathcal{P}_{N}} \frac{a_{n}}{n^{1/k}} n^{-s} \in \mathcal{H}^{2,(N)}$ for every $N$. Moreover,
\[
\Big\Vert \textstyle \sum_{n \in \mathcal{P}_{N}} \frac{a_{n}}{n^{1/k}} n^{-s} \Big\Vert_{\mathcal{H}^{2}}
= \big\Vert \textstyle \sum_{n \in \mathcal{P}_{N}} a_{n} n^{-s} \Big\Vert_{\mathcal{H}^{2}_{k}}
= \Vert f_{N} \Vert_{H^{2} (\mathfrak{p}^{-1/k} \mathbb{D}^{N})}
\leq \Vert f \Vert_{H^{2} (\ell_{2} (\mathfrak{p}^{1/k}) \cap B_{\ell_{\infty} (\mathfrak{p}^{1/k}) })} \,.
\]
With this \cite[Corollary~13.9]{defant2018Dirichlet} gives $\sum \frac{a_{n}}{n^{1/k}} n^{-s} \in \mathcal{H}^{2}$. Then $\sum a_{n}n^{-s} \in \mathcal{H}^{2}_{k}$ and also $\Vert \sum a_{n}n^{-s} \Vert_{2,k} \leq \Vert f \Vert_{H^{2} (\ell_{2} (\mathfrak{p}^{1/k}) \cap B_{\ell_{\infty} (\mathfrak{p}^{1/k}) })}$.\\
Take now $\sum a_{n} n^{-s} \in \mathcal{H}^{2}_{k}$ and define $c_{\alpha} = a_{\mathfrak{p}^{\alpha}}$ for each $\alpha \in \mathbb{N}_{0}^{(\mathbb{N})}$. For $z \in \ell_{2} (\mathfrak{p}^{1/k}) \cap B_{\ell_{\infty} (\mathfrak{p}^{1/k}) }$ we have
\begin{equation}\label{haendel}
\sum_{\alpha \in \mathbb{N}_{0}^{(\mathbb{N})}} \vert c_{\alpha} z^{\alpha} \vert
= \sum_{\alpha \in \mathbb{N}_{0}^{(\mathbb{N})}} \frac{\vert c_{\alpha} \vert}{(\mathfrak{p}^{1/k})^{\alpha}}  \big\vert (\mathfrak{p}^{1/k})^{\alpha} z^{\alpha} \big\vert
\leq \bigg( \sum_{\alpha \in \mathbb{N}_{0}^{(\mathbb{N})}} \frac{\vert c_{\alpha} \vert^{2}}{(\mathfrak{p}^{\alpha})^{2/k}} \bigg)^{\frac{1}{2}}
\bigg( \sum_{\alpha \in \mathbb{N}_{0}^{(\mathbb{N})}} \big\vert \mathfrak{p}^{2/k} z^{2} \big\vert^{\alpha} \bigg)^{\frac{1}{2}} \,.
\end{equation}
Observe that $\big(\mathfrak{p}_{n}^{2/k} z_{n}^{2} \big)_{n} \in \ell_{1} \cap B_{c_{0}}$ and therefore the last sum in \eqref{haendel} converges (see e.g. \cite[Remark~2.18]{defant2018Dirichlet}). On the other hand, $\sum_{n=1}^{\infty} \frac{\vert a_{n} \vert^{2}}{n^{2/k}} < \infty$ (because the Dirichlet series belongs to $\mathcal{H}^{2}_{k}$) and, so, $\sum_{\alpha \in \mathbb{N}_{0}^{(\mathbb{N})}} \frac{\vert c_{\alpha} \vert^{2}}{(\mathfrak{p}^{\alpha})^{2/k}}$ also converges. This altogether shows that the power series in \eqref{haendel} converges (absolutely) and, then $f(z) = \sum_{\alpha \in \mathbb{N}_{0}^{(\mathbb{N})}}  c_{\alpha} z^{\alpha}$ defines a holomorphic function on $\ell_{2} (\mathfrak{p}^{1/k}) \cap B_{\ell_{\infty} (\mathfrak{p}^{1/k}) }$. \cite[Theorem~15.57]{defant2018Dirichlet}. \\
Consider for each $N$ the restriction of $f$ to $\mathfrak{p}^{-1/k} \mathbb{D}^{N}$ and denote it by $f_{N}$. This belongs to $H^{2}(\mathfrak{p}^{-1/k} \mathbb{D}^{N})$ (see Proposition~\ref{praetroius}) and
\[
\Vert f_{N} \Vert_{H^{2}(\mathfrak{p}^{-1/k} \mathbb{D}^{N})}
= \big\Vert \textstyle \sum_{n \in \mathcal{P}_{N}} \frac{a_{n}}{n^{-1/k}} n^{-s} \big\Vert_{\mathcal{H}_{2}}
\leq \big\Vert \textstyle \sum \frac{a_{n}}{n^{-1/k}} n^{-s} \big\Vert_{\mathcal{H}_{2}}\,.
\]
This immediately yields $f \in H^{2} (\ell_{2} (\mathfrak{p}^{1/k}) \cap B_{\ell_{\infty} (\mathfrak{p}^{1/k}) })$ with $\Vert f \Vert \leq \big\Vert \sum a_{n} n^{-s} \big\Vert_{2,k}$ and completes the proof.
\end{proof}

\begin{cor}\label{identificacion}
\[
\mathcal{H}_{+} = \bigcap_{k=1}^{\infty} H^{2} (\ell_{2} (\mathfrak{p}^{1/k}) \cap B_{\ell_{\infty} (\mathfrak{p}^{1/k}) }) \,,
\]
where in the intersection the topology of the projective limit is considered.
\end{cor}

\section{Composition operators on the space $\mathcal{H}_{+}$} \label{sect:compop}
	
If $\phi : \mathbb{C}_{\theta} \to \mathbb{C}_{\mu}$ is holomorphic and $D = \sum a_{n} n^{-s} \in \mathcal{D}$ converges on $\mathbb{C}_{\mu}$, then one can consider the composition $D \circ \phi = \sum a_{n} n^{-\phi(s)}$. This defines a 
holomorphic function on $\mathbb{C}_{\theta}$. Following the work performed in  \cite{gordon1999composition,bayart2002hardy,bonet2018frechet}, in this section we analyse those 
functions $\phi : \C_{\theta}\rightarrow\C_{\frac{1}{2}}$ (since every Dirichlet series in $\mathcal{H}_{+}$ converges at least on $\mathbb{C}_{\frac{1}{2}}$ we are led to consider functions taking values in this half-plane)	
whose associated composition operators are  bounded or continuous. We begin by looking at the case $\mathcal{H}_{+} \to \mathcal{D}$, that goes essentially along the same lines as \cite[Theorem~A]{gordon1999composition}.  

\begin{teo} \label{juliososa}
A function $\phi:\C_{\theta}\rightarrow\C_{\frac{1}{2}}$ defines a composition operator $C_{\phi}:\mathcal{H}_{+}\rightarrow\mathcal{D}$ if and only if it is an analytic function on some half-plane $\mathbb{C}_{\mu}$ and there it has the form
\begin{equation} \label{fi}
		\phi(s)=c_{0}s+\varphi(s), \text{ where } c_{0}\in\mathbb{N}\cup\{0\} \text{ and } \varphi\in\mathcal{D} \,.
\end{equation} 
\end{teo}
\begin{proof}
If $C_{\phi}$ is a composition operator, then $2^{-\phi(s)}$ and $3^{-\phi(s)}$ are holomorphic and nowhere-vanishing functions on some half-plane  $\mathbb{C}_{\mu}$ with $\mu\geq\theta$, so with the same proof given in \cite[Lemma 2.1]{bayart2019composition} we have that  $\phi$ is an analytic function on $\mathbb{C}_{\mu}$. On the other hand, the fact that $\mathcal{H}^{2}\subseteq\mathcal{H}_{+}$  and \cite[Theorem~A]{gordon1999composition} immediately give that $\phi :\mathbb{C}_{\mu} \to \mathbb{C}_{\frac{1}{2}}$ has to be as in \eqref{fi}.\\
Suppose now that $\phi(s)=c_{0}s+\varphi(s)$ with $c_{0} \in \mathbb{N}\cup\{0\}$ and $\varphi \in \mathcal{D}$. Then $\phi$ generates a composition operator on $\mathcal{H}^{2}$. For each fixed $\varepsilon>0 $ we define  $\phi_{\varepsilon}(s)=c_{0}s+\varphi(s)-\varepsilon$, which also defines a composition operator on $\mathcal{H}^{2}$. Given a Dirichlet series $D=\sum{a_{n}n^{-s}}$ in $\mathcal{H}_{+}$  we have 
\[
		D \circ \phi (s) = \sum{a_{n} n^{-\phi(s)}} = \sum{\frac{a_{n}}{ n^{\varepsilon}} n^{-(\phi(s)-\varepsilon)}}= D_{1/\varepsilon} \circ \phi_{\varepsilon}(s) \in \mathcal{D}
\] 
\end{proof}
	
We move on now to characterise the continuous composition operators on $\mathcal H_{+}$, following the spirit of \cite[Theorem~B]{gordon1999composition}.
	
\begin{teo}\label{teo: op cont}
Let $\phi$ be as in \eqref{fi}. Then $C_{\phi}:\mathcal{H}_{+}\rightarrow\mathcal{H}_{+}$ defines a continuous composition operator if and only if $\phi$ has an analytic extension to $\mathbb{C}_{+}$, such that
\begin{enumerate}
	\item $\phi (\mathbb{C}_{+}) \subseteq \mathbb{C}_{+}$ if $c_{0}\in\mathbb{N}$,
	\item $\phi (\mathbb{C}_{+}) \subseteq \mathbb{C}_{\frac{1}{2}}$ if $c_{0}=0$.
\end{enumerate}
\end{teo}

Let us make a short comment before we proceed to the proof. Suppose  $\phi : \mathbb{C}_{+} \to \mathbb{C}$ is a holomorphic mapping as in \eqref{fi} that defines a composition operator $C_{\phi} : \mathcal{H}_{+} \to \mathcal{D}$. Given $k \in \mathbb{N}$ and $\delta >0$ we define a function $\phi_{k,\delta} : \mathbb{C}_{-\frac{1}{k}} \to \mathbb{C}$ by
\begin{equation} \label{fiks}
	\phi_{k,\delta} (s) = \phi \big( s + \tfrac{1}{k} \big) - \delta \,.
\end{equation}
This is obviously holomorphic, but even more
\begin{equation} \label{1019}
	\phi_{k,\delta} (s) = c_{0} \big( s + \tfrac{1}{k} \big) + \varphi \big( s + \tfrac{1}{k} \big) - \delta
	= c_{0}s + \varphi_{k,\delta} (s) \,,
\end{equation}
where 
\[
	\varphi_{k,\delta} (s) =  \frac{c_{0}}{k}  -  \delta  + \varphi_{k}(s) 
	=  \frac{c_{0}}{k}  -  \delta  + c_{1} + \sum_{n=2}^{\infty} \frac{a_{n}}{n^{1/k}} n^{-s}
\]
is a Dirichlet series that converges at least in the same half-plane as $\varphi$ (and that, as $\varphi$, has a holomorphic extension to $\mathbb{C}_{+}$). Then, by Theorem~\ref{juliososa}, $\phi_{k,\delta}$ defines a composition operator $C_{\phi_{k,\delta}} : \mathcal{H}_{+} \to \mathcal{D}$. Let us note that, if $D = \sum a_{n} n^{-s} \in \mathcal{H}_{+}$, then $D \circ \phi \in \mathcal{D}$ and, whenever this converges, for $k \in \mathbb{N}$ we have (recall \eqref{charpentier})
\[
	\big( D \circ \phi \big)_{k} (s) =  D \circ \phi  \big( s + \tfrac{1}{k} \big)
	= \sum a_{n} n^{-\phi  ( s + 1/k )} = \sum \frac{a_{n}}{n^{\delta}} n^{-\phi  ( s + 1/k ) + \delta}
	= D_{1/\delta} \circ \phi_{k,\delta} (s) \,,
\]
that is (note that $D_{1/\delta}$ again belongs to $\mathcal{H}_{+}$)
\begin{equation} \label{franck}
	\big( D \circ \phi \big)_{k} = C_{\phi_{k,\delta}} (D_{1/\delta} ) \,.
\end{equation}
	
\begin{proof}[proof of Theorem~\ref{teo: op cont}]
Take $\phi$ such that $C_{\phi}:\mathcal{H}_{+}\rightarrow\mathcal{H}_{+}$ is continuous. For each fixed $k$ we consider $\phi_{k,0}$ as in \eqref{fiks} and observe that it defines a 
continuous composition operator $\mathcal{H}^{2} \to \mathcal{H}^{2}$ and (as in \eqref{franck}) $C_{\phi_{k,0}} (D) = \big( D \circ \phi \big)_{k}$ for every $D \in \mathcal{H}^{2}$. Suppose that
$c_{0} =0$; then by  \cite{gordon1999composition} $\phi_{k} (\mathbb{C}_{+}) \subseteq \mathbb{C}_{\frac{1}{2}}$ which yields $\phi (\mathbb{C}_{\frac{1}{k}}) \subseteq 
\mathbb{C}_{\frac{1}{2}}$. Since this holds for every $k$ the conclusion follows. The same argument gives also the conclusion for  $c_{0} \neq 0$. This completes the proof of necessity.\\
To prove sufficiency, let us note in first place that if $\phi$ is such that for every $k$ there is $\delta >0$ so that the function in \eqref{fiks} defines a continuous composition operator $C_{\phi_{k,\delta}} : \mathcal{H}^{2} \to \mathcal{H}^{2}$, then we just have to choose $m$ with $0 < \frac{1}{m} < \delta$ and take \eqref{franck} into account to get
\[ 
\Vert C_{\phi} (D) \Vert_{2,k} = \Vert ( D \circ \phi)_{k} \Vert_{2} 
\leq \Vert C_{\phi_{k,\delta}} \Vert \, \Vert D_{1/\delta} \Vert_{2}
\leq \Vert C_{\phi_{k,\delta}} \Vert \, \Vert D \Vert_{2,m} 
\]
for every $D \in \mathcal{H}_{+}$; and this shows that $C_{\phi} : \mathcal{H}_{+} \to \mathcal{H}_{+}$ is continuous. It is then enough, then to show that if $\phi$ satisfies any of the conditions in Theorem~\ref{teo: op cont}, then for each $k$ we can find $\delta >0$ so that the operator $C_{\phi_{k,\delta}} : \mathcal{H}^{2} \to \mathcal{H}^{2}$ is continuous. We consider different cases. 
First of all, if $c_{0} =0$, then $\phi = \varphi$ and, by hypothesis $\varphi (\mathbb{C}_{+}) = \phi (\mathbb{C}_{+})  \subseteq \mathbb{C}_{\frac{1}{2}}$. Given $k \in \mathbb{N}$, by \cite[Proposition~4.2]{gordon1999composition} (see Remark~\ref{gordon4.2}) we can find $\delta >0$ so that $\varphi (\mathbb{C}_{\frac{1}{k}}) \subseteq \mathbb{C}_{\frac{1}{2}+\delta}$. Taking this $\delta$ we define $\phi_{k,\delta}$ as in \eqref{fiks}. A simple computation shows that $\phi_{k,\delta} (\mathbb{C}_{+}) \subseteq \mathbb{C}_{\frac{1}{2}}$ and \cite[Theorem~B]{gordon1999composition} gives that $C_{\phi_{k,\delta}} : \mathcal{H}^{2} \to \mathcal{H}^{2}$ is continuous. This completes the proof in this case.\\
Suppose now that $c_{0} \neq 0$ and $\varphi \neq 0$. In this case $\phi (\mathbb{C}_{+})  \subseteq \mathbb{C}_{+}$ and, by \cite[Proposition~4.2]{gordon1999composition}, given $k \in \mathbb{N}$ we can find $\delta >0$ so that $\varphi (\mathbb{C}_{\frac{1}{k}}) \subseteq \mathbb{C}_{\delta}$. Then $\phi_{k,\delta} (\mathbb{C}_{+}) \subseteq \mathbb{C}_{+}$ and again \cite[Theorem~B]{gordon1999composition} yields that $C_{\phi_{k,\delta}} : \mathcal{H}^{2} \to \mathcal{H}^{2}$ is continuous, giving the result.\\
The remaining case ($c_{0} \neq 0$ and $\varphi = 0$) follows by a direct computation. Note that in this case $\phi (s) = c_{0}s$ and, for each fixed $k$, we have
\[
	\Vert C_{\phi} (D) \Vert_{2,k} 
	= \Big\Vert \sum a_{n} n^{- c_{0}s}\Big\Vert_{2,k}
	= \Big\Vert \sum \tfrac{a_{n}}{n^{c_{0}/k}} \big( n^{c_{0}} \big)^{-s}\Big\Vert_{2}
	= \sum_{n=1}^{\infty} \frac{\vert a_{n} \vert ^{2} }{n^{2c_{0}/k}}
	\leq \sum_{n=1}^{\infty} \frac{\vert a_{n} \vert ^{2} }{n^{2/k}} 
	= \Vert D \Vert_{2,k}
\]
for every $D=\sum a_{n} n^{-s} \in \mathcal{H}_{+}$.
\end{proof}

An operator $T:E \to F$ between locally convex spaces is said to be bounded if there is a $0$-neighbourhood $U$ in $E$ for which $T(U)$ is bounded (in $F$).

\begin{obs} \label{convergencia}
Suppose that $(D^{N})_{N}$ is a sequence of Dirichlet series (say $D^{N} = \sum a_{n}^{(N)} n^{-s}$) in $\mathcal{H}^{2}$ that converges to some $D =\sum a_{n} n^{-s} \in \mathcal{H}^{2}$. Then for each $s \in \mathbb{C}_{\frac{1}{2}}$ we have
\[
		\vert D^{N} (s) - D(s) \vert \leq \sum_{n=1}^{\infty} \vert a_{n}^{(N)} - a_{n} \vert \frac{1}{n^{\re s}}
		\leq \bigg(  \sum_{n=1}^{\infty} \vert a_{n}^{(N)} - a_{n} \vert^{2} \bigg)^{\frac{1}{2}} 
		\bigg(  \sum_{n=1}^{\infty} \frac{1}{n^{2 \re s}}\bigg)^{\frac{1}{2}} \,,
\]
and  $D^{N} (s) \to D(s)$ as $N \to \infty$.
\end{obs}

\begin{teo}
Let $\phi$, defined as in \eqref{fi}, be such that the composition operator $C_{\phi}:\mathcal{H}_{+}\rightarrow\mathcal{H}_{+}$ is continuous. Then $C_{\phi}$ is a bounded composition operator in $\mathcal{H}_{+}$ if and only if there exists $\varepsilon >0$ such that
\begin{enumerate}
		\item $\phi(\C_{+})\subseteq\C_{\varepsilon}$  if $c_{0}\in\mathbb{N}$,
		\item $\phi(\C_{+})\subseteq\C_{\frac{1}{2}+\varepsilon}$  if $c_{0}=0$.
\end{enumerate}
\end{teo}
\begin{proof}
Let us begin by assuming that $c_0\in\mathbb{N}$ and $\phi(\C_{+})\subseteq\C_{\varepsilon}$ for some $\varepsilon>0$. Take $0 < \frac{1}{m}< \varepsilon$ and consider the following neighbourhood of $0$ in $\mathcal{H}_{+}$, 
\[
		U_{m}= \{ D\in\mathcal{H}_{+} \, \colon \, \Vert D_{m}\Vert_2<1 \} \,.
\]
For each $k \in \mathbb{N}$ we consider the function $\phi_{k,1/m}$ defined in \eqref{fiks} and note that, if $\re s >0$ we have
\[
\re \phi_{k,1/m} (s) = \re \phi \big(s + \tfrac{1}{k} \big) - \frac{1}{m} > \varepsilon - \frac{1}{m} >0 \,.
\]
In other words, $\phi_{k,1/m} (\mathbb{C}_{+}) \subseteq \mathbb{C}_{+}$ and, by \cite[Theorem~B]{gordon1999composition}, the composition operator $C_{\phi_{k,1/m}} : \mathcal{H}^{2} \to \mathcal{H}^{2}$ is continuous. Having this,  if $D \in U_{m}$ (note that $D_{m} \in \mathcal{H}^{2}$ and recall \eqref{franck}), then 
\[
	\Vert C_{\phi} (D) \Vert_{2,k} = \Vert C_{\phi_{k,1/m}} (D_{m}) \Vert_{2} 
	\leq \Vert C_{\phi_{k,1/m}} \Vert \, \Vert D_{m} \Vert_{2}
	= \Vert C_{\phi_{k,1/m}} \Vert \, \Vert D \Vert_{2,m}
	\leq \Vert C_{\phi_{k,1/m}} \Vert \,,
\]
and $C_{\phi}$ is bounded. The case $c_{0}=0$ and $\phi(\C_{+})\subseteq\C_{\frac{1}{2}+\varepsilon}$ follows in the same way.\\
Suppose now that $\phi$ generates a bounded operator. Then there exists $N\in\mathbb{N}$ (that we may assume $N\geq 2$, since the sequence of seminorms $\Vert \cdot \Vert_{2,k}$ is increasing) such that for each $k\in\mathbb{N}$ there is $M_{k}>0$ such that $\Vert C_{\phi}(D)\Vert_{2,k}\leq{M_{k}\Vert D\Vert_{2,N}}$ for every $D\in\mathcal{H}_{+}$. Fix now some $k \in \mathbb{N}$ and consider the function $\phi_{k,1/N}$ defined in \eqref{fiks}. By Theorem~\ref{juliososa} (check \eqref{1019}) it defines a composition operator $\mathcal{H}_{+} \to \mathcal{D}$, and our aim now is to show that in fact $C_{\phi_{k,1/N}} : \mathcal{H}^{2} \to \mathcal{H}^{2}$ is continuous.\\
If $P = \sum_{n=1}^{M} a_{n} n^{-s}$ is a Dirichlet polynomial, we may consider $P_{-N}=\sum_{n=1}^{M} a_{n}n^{\frac{1}{N}}n^{-s}$. This is again a Dirichlet polynomial, and as in \eqref{franck} we have $C_{\phi_{k,1/N}} (P) = (P_{-N} \circ \phi)_{k}$. Hence
\[
		\Vert C_{\phi_{k,1/N}} (P) \Vert_{2} = \Vert P_{-N} \circ \phi \Vert_{2,k}
		\leq M_{k} \Vert P_{-N}\Vert_{2,N} = M_{k} \Vert P \Vert_{2} \,.
\]
This shows that $C_{\phi_{k,1/N}}$ defines on the Dirichlet polynomials a continuous operator taking values in $\mathcal{H}^{2}$. This operator extends by density to a continuous $G_{k,N} : \mathcal{H}^{2} \to \mathcal{H}^{2}$. Our aim now is to see that this operator in fact coincides with $C_{\phi_{k,1/N}}$. Choose some $D \in \mathcal{H}^{2}$ and take a sequence $(P_{M})_{M}$ of Dirichlet polynomials converging in $\mathcal{H}^{2}$ to $D$. Then $G_{k,N} (P_{M}) \to G_{k,N}(D)$. Since $C_{\phi}$ is
continuous we know from Theorem~\ref{teo: op cont} that $\phi (\mathbb{C}_{+}) \subseteq \mathbb{C}_{+}$.
Then, by \cite[Proposition~4.2]{gordon1999composition} (see Remark~\ref{gordon4.2}), $\varphi (\mathbb{C}_{1/k}) \subseteq \mathbb{C}_{\sigma}$. Let us see now that, if $\re s >1$, then
\begin{equation} \label{brl}
		\re  \Big( \phi_{k,1/N} (s) \Big) 
		= c_{0} \re s + \frac{c_{0}}{k} - \frac{1}{N} + \re \Big( \varphi \big( s + \tfrac{1}{k} \big) \Big) > \frac{1}{2} \,.
\end{equation}
Indeed, the case $c_{0}\neq0$ is clear. On the other hand, if $c_0=0$ then $\varphi(\mathbb{C}_1)\subset\mathbb{C}_{1/2 + \varepsilon}$ for some $\varepsilon >0$ and therefore taking $N$ sufficiently large, we have the inequality.\\
Once we have established \eqref{brl}, Remark~\ref{convergencia} implies that $P_{M} \big(\phi_{k,1/N} (s) \big) \to D  \big(\phi_{k,1/N} (s) \big)$. Hence $G_{k,N}(D)$ and $D \circ \phi_{k,1/N}$ coincide on $\mathbb{C}_{1}$, and they have to be equal in all the half-plane where they are defined. Therefore $G_{k,N} = C_{\phi_{k,1/N}}$, and the composition operator is continuous from $\mathcal{H}^{2}$ to $\mathcal{H}^{2}$. Then, by \cite[Theorem~B]{gordon1999composition}, $\phi_{k,1/N} (\mathbb{C}_{+}) \subseteq \mathbb{C}_{\sigma}$, where $\sigma = 0$ if $c_{0} \neq 0$ and $\sigma = \frac{1}{2}$ if $c_{0}=0$. This gives $\phi (\mathbb{C}_{\frac{1}{k}}) \subseteq \mathbb{C}_{\sigma + \frac{1}{N}}$ and, since $k$ was arbitrary, yields our claim.
\end{proof}

We finish this section by finding conditions so that the operator takes values in a smaller space, in the spirit of \cite[Theorems~12 and ~13]{bayart2002hardy}, where conditions where given so that $\mathcal{H}^{p}$ is mapped to $\mathcal{H}^{q}$ for $1 \leq q \leq \infty$.

Our first result in this matter gives sufficient conditions on the composition operator to take values in $\mathcal{H}^{p}$. Its proof is modelled along the same lines as \cite[Theorem~12]{bayart2002hardy}
	
\begin{teo}\label{teo: inclusion en Hp}
Let $\phi(s)=c_{0}s+\varphi(s)$, with $c_{0}\geq1$ and $\varphi\in\mathcal{D}$ be such that $\phi(\mathbb{C}_{+})\subseteq{\mathbb{C}_{\varepsilon}}$ for some $\varepsilon>0$. Then $C_{\phi} (\mathcal{H}_{+}) \subseteq \mathcal{H}^{p}$ for every $1\leq{p}<\infty$.
\end{teo}
\begin{proof}
Let us consider the function
\[
		\gamma (s) = \phi(s) - \varepsilon = c_{0}s + \varphi (s) - \varepsilon = c_{0}s+ \psi(s) \,,
\]	
and note that $\psi = (c_{1} - \varepsilon) + \sum_{n \geq 2} c_{n} n^{-s} \in \mathcal{D}$. On the one hand we have $\phi(\mathbb{C}_{+})\subseteq\mathbb{C}_{\varepsilon}$. This, by \cite[Proposition~4.2]{gordon1999composition} (see Remark~\ref{gordon4.2}), gives $\varphi(\mathbb{C}_{\frac{1}{2}}) \subseteq \mathbb{C}_{\varepsilon + \delta}$, which together with the fact that $c_{0}\geq 1$, yields
$\phi(\mathbb{C}_{\frac{1}{2}})\subseteq{\mathbb{C}_{\frac{1}{2}+\varepsilon}}$. Then $\gamma(\mathbb{C}_{\frac{1}{2}})\subseteq \mathbb{C}_{\frac{1}{2}}$ and $\gamma(\mathbb{C_{+}})\subseteq \mathbb{C}_{+}$. By  \cite[Theorem~11]{bayart2002hardy} $C_{\gamma} : \mathcal{H}^{p} \to \mathcal{H}^{p}$ is a well defined composition operator for every $2 \leq p < \infty$. Fix some $2 \leq p < \infty$.  Given $D = \sum a_{n} n^{-s} \in \mathcal{H}_{+}$  we have $D_{\frac{1}{\varepsilon}} \in \mathcal{H}^{p}$  and
\[
		C_{\phi}(D) = \sum a_{n} n^{- \phi(s)} =  \sum \frac{a_{n}}{n^{\varepsilon}} n^{- \phi(s)  +\varepsilon} 
		=   \sum \frac{a_{n}}{n^{\varepsilon}} n^{- \gamma(s) } = C_{\gamma} (D_{\frac{1}{\varepsilon}}) \in \mathcal{H}^{p} \,.
\]
This and the fact that $\mathcal{H}^{2} \subseteq \mathcal{H}^{p} \subseteq \mathcal{H}^{1}$ for every $1 \leq p \leq 2$ complete the proof.
\end{proof}

\begin{teo}
Let $\phi:\mathbb{C}_{+}\to\mathbb{C}_{\frac{1}{2}}$  be as in \eqref{fi}. Then 
\begin{enumerate}
	\item $C_{\phi} (\mathcal{H}_{+}) \subseteq \mathcal{H}^{\infty}$ if and only if $\phi(\mathbb{C}_{+})\subseteq\mathbb{C}_{\frac{1}{2}+\varepsilon}$ for some $\varepsilon>0$,
	\item $C_{\phi} (\mathcal{H}_{+} ) \subseteq \mathcal{H}^{\infty}_{+}$ if and only if $\phi(\mathbb{C}_{+})\subseteq\mathbb{C}_{\frac{1}{2}}$.
\end{enumerate}
\end{teo}
\begin{proof}
First of all, if $C_{\phi} (\mathcal{H}_{+}) \subseteq \mathcal{H}_{\infty}$ then, in particular (recall Remark~\ref{inclusiones1}), $\C_{\phi}(\mathcal{H}^{p})\subseteq{\mathcal{H}_{\infty}}$ and \cite[Theorem~13]{bayart2002hardy} gives $\phi(\mathbb{C}_{+})\subseteq\mathbb{C}_{\frac{1}{2}+\varepsilon}$ for some $\varepsilon>0$.\\
Let us conversely suppose that there is $\varepsilon >0$ so that $\re(\phi(s))>\frac{1}{2}+\varepsilon$ for all $s\in\mathbb{C}_{+}$. Take a Dirichlet series $D=\sum{a_{n}n^{-s}}$ in $\mathcal{H}_{+}$ and $s\in\mathbb{C}_{+}$, then
\[
		|C_{\phi}(D)(s)|= | \sum_{n=1}^{\infty} a_{n}n^{-\phi(s)}|
		\leq \sum_{n=1}^{\infty} |a_{n}| n^{-\frac{1}{2}-\varepsilon}
		\leq \bigg( \sum_{n=1}^{\infty} \frac{|a_{n}|^{2}}{n^{\varepsilon}} \bigg)^{\frac{1}{2}} \bigg( \sum_{n=1}^{\infty} 
		\frac{1}{n^{1+\varepsilon}} \bigg)^{\frac{1}{2}} \,.
\]
The last term is finite because $D\in\mathcal{H}_{+}$ (and obviously does not depend on $s$). Hence $C_{\phi}(D)\in\mathcal{H}_{\infty}$ and the proof of the first statement is completed. \\
For the second statement let us assume that $\phi(\mathbb{C}_{+})\subseteq\mathbb{C}_{\frac{1}{2}}$, and fix $\varepsilon >0$. From \cite[Proposition~4.2]{gordon1999composition} we know that we can find $\delta >0$ so that $\phi(\mathbb{C}_{\varepsilon}) \subseteq \mathbb{C}_{\frac{1}{2} + \delta}$. 
Define a function
\[
	\Psi (s) = \phi(s+ \varepsilon) =c_{0} s+ c_{0} \varepsilon + \varphi_{\frac{1}{\varepsilon}}(s) 
\]
and observe that $\psi  = c_{0} \varepsilon + \varphi_{\frac{1}{\varepsilon}} \in \mathcal{D}$. Clearly $\Psi(\mathbb{C}_{+} ) = \phi ( \mathbb{C}_{\varepsilon} )\subseteq \mathbb{C}_{\frac{1}{2}+ \delta}$ and, therefore
the composition operator $C_{\Psi}$ maps $\mathcal{H}_{+}$ into $\mathcal{H}_{\infty}$. For $D \in \mathcal{H}_{+}$ we have
\[
	\sup_{s\in\mathbb{C}_{\varepsilon}} |C_{\phi}(D)(s)| 
	= \sup_{s\in\mathbb{C}_{\varepsilon}} |D(\phi(s))| 
	=\sup_{s\in\mathbb{C}_{+}} |D(\phi(s+\varepsilon))| 
	=\sup_{s\in\mathbb{C}_{+}} | C_{\Psi} (D)| \,.
\]
Being the latter finite (because $C_{\Psi} (D) \in \mathcal{H}^{\infty}$) and $\varepsilon >0$ arbitrary this yields $C_{\phi}(D) \in \mathcal{H}^{\infty}_{+}$. Necessity follows just by hypothesis.
\end{proof}

We finish this section by looking at vertical limits, a topic that in our view is interesting by itself and which eventually will lead to an alternative (though longer) proof of the necessity in Theorem~\ref{teo: op cont}. We begin by fixing some concepts and notations. First of all, a character is a function $\chi : \mathbb{N} \to \mathbb{T}$ satisfying $\chi(nm) = \chi (n) \chi(m)$ for every $n,m$. We 
write $\Xi$ for the group of all characters that, as a matter of fact, is the dual group of the multiplicative group $\mathbb{Q}_{+}$ of positive rationals. With this in mind $\Xi$ is a compact abelian group that has a unique (up to normalisation) Haar measure. In \cite{hedenmalm1995hilbert} it was shown how $\Xi$ can be identified with $\mathbb{T}^{\infty}$ and the Haar measure corresponds to the product of the normalised Lebesgue measure.\\
Given a Dirichlet series $D= \sum a_{n} n^{-s}$ and a character $\chi \in \Xi$ we denote
\[
	D_{\chi} = \sum a_{n} \chi(n) n^{-s} \,.
\]
We know from \cite[Lemma~2.4]{hedenmalm1995hilbert} that the functions $D_{\chi}$ for $\chi \in \Xi$ correspond exactly with the vertical limits of $D$. 
Clearly, if $D$ belongs to $\mathcal{D}$ (or to $\mathcal{H}_{+}$), then so also does $D_{\chi}$. Now, if $\phi$ is defined as in \eqref{fi}, then we write
\[
	\phi_{\chi} (s) = c_{0} s + \varphi_{\chi} \,,
\]
which can no longer be seen as a vertical limit. Our first step is to see, following \cite[Proposition~4.3]{gordon1999composition} how this action of the characters relates with composition. Let us recall that if $\phi$ is as in \eqref{fi}, by \cite[equation (4.2)]{gordon1999composition}, we have
\begin{equation} \label{monomio}
	\big( n^{- \phi} \big)_{\chi} (s) = \chi(n)^{c_{0}} n^{- \phi_{\chi}(s)} 
\end{equation}
for $s \in \mathbb{C}_{\sigma_{u}(\varphi)}$.

\begin{obs} \label{gordon4.2}
If $\phi$ is defined as in \eqref{fi}, then \cite[Proposition~4.2]{gordon1999composition} provides us with a good control of the behaviour of $\varphi$ in half-planes. More precisely, if $\phi ( \mathbb{C}_{\theta}) \subseteq \mathbb{C}_{\eta}$, then for every $\varepsilon >0$ there is some $\delta >0$ so that $\varphi (\mathbb{C}_{\theta + \varepsilon}) \subseteq \mathbb{C}_{\eta + \delta - c_{0} \theta}$.
\end{obs}

\begin{prop}\label{prop: 3}
Suppose $\phi:\C_{\theta}\rightarrow\C_{\frac{1}{2}}$ is a holomorphic mapping as in \eqref{fi}, 
then for $D\in\mathcal{H}_{+}$ and $\chi\in\Xi$, the following relation holds:
\begin{equation} \label{vdpoel}
		(D\circ\phi )_{\chi}(s) = D_{\chi^{c_{0}}} \circ \phi_{\chi}(s) 
\end{equation}
for $s\in\C_{\theta}$.
\end{prop}
\begin{proof}
On the right hand side, by \cite[Proposition~4.1]{gordon1999composition}, we have that $\phi_{\chi}$ maps the half-plane $\mathbb{C}_{\theta}$ into $\mathbb{C}_{\frac{1}{2}}$. On the other hand, $D_{\chi^{c_{0}}}$ is in $\mathcal{H}_{+}$, since $D\in\mathcal{H}_{+}$ and $\chi^{c_{0}} \in \Xi$ (because $c_{0} \in \mathbb{N}_{0}$). Then (see \eqref{sigma-aes}) it converges absolutely on $\mathbb{C}_{\frac{1}{2}}$ , and $D_{\chi^{c_{0}}} \circ \phi_{\chi}$ defines a holomorphic function on $\mathbb{C}_{\theta}$.\\
On the left hand side, since $D = \sum a_{n} n^{-s} \in \mathcal{H}_{+}$, it converges absolutely on $\mathbb{C}_{\frac{1}{2}}$ (see again \eqref{sigma-aes}).
Then 
\[
	D\circ\phi(s)=\sum\limits_{n=1}^{\infty}{a_{n}n^{-\phi(s)}}
	= \sum_{n=1}^{\infty}  a_{n}  n^{-c_{0}s-c_1} \prod_{k=2}^{\infty} \left(1+\sum_{j=1}^{\infty} \frac{(-c_k\log(n))^{j}}{j!}k^{-js}\right) 
\]
can be rearranged as a Dirichlet series on some $\mathbb{C}_{\theta}$ for $\theta$ big enough.\\		 
Then, $(D\circ\phi)_{\chi}$ is a Dirichlet series that converges absolutely (an then defines a holomorphic function) on $\mathbb{C}_{\theta}$. The idea now is to see that these two functions coincide. We distinguish several cases.\\
First of all, if $\varphi$ is constant, then $\phi =\phi_{\chi} = c_{0}s + c_1$, and
\begin{multline*}
		(D \circ \phi )_{\chi} (s) =\bigg( \sum \big( a_{n} n^{-c_{1}} \big) (n^{c_{0}})^{-s} \bigg)_{\chi} \\
		= \sum \big( a_{n} n^{-c_{1}} \big) \chi \big( n^{c_{0}}\big) (n^{c_{0}})^{-s} 
		=\sum{\chi(n)^{c_{0}}a_{n}n^{-\phi_{\chi}(s)}}
		=D_{\chi^{c_{0}}}\circ\phi_{\chi}(s)
\end{multline*}
for every $s \in \mathbb{C}_{\theta}$, which gives our claim.\\
Suppose now that $\varphi$ is not constant and fix $\theta' > \sigma_{u}(\varphi) \geq \theta$. 
By  \cite[Proposition~4.2]{gordon1999composition} (see Remark~\ref{gordon4.2}), we can find some $\varepsilon=\varepsilon(\theta')>0$ such that $\varphi \big( \mathbb{C}_{\theta'} \big) \subseteq \mathbb{C}_{\frac{1}{2}+\varepsilon-c_{0}\theta}$ and
\[
	\re(\phi(s))=c_{0}\re(s)+\re(\varphi(s)) \geq c_{0} \theta + \frac{1}{2} + \varepsilon- c_{0}\theta =\frac{1}{2}+\varepsilon
\]
for every $s\in\mathbb{C}_{\theta'}$. Hence
\[
	\Vert n^{-\phi}\Vert_{\mathbb{C}_{\theta'}} : =\sup_{s \in \mathbb{C}_{\theta'}} \vert n^{-\phi(s)} \vert \leq \frac{1}{n^{\frac{1}{2}+\varepsilon}} .
\]
For each $N$ let $D^{(N)} = \sum_{n=1}^{N} a_{n} n^{-s}$ denote the $N$-th partial sum of $D$. Fix $k$ with $0 < \frac{1}{k} < \varepsilon$, then
\[
	\sum |a_{n}|\;\Vert n^{-\phi} \Vert_{\mathbb{C}_{{\theta}'}} 
	\leq \sum |a_n| n^{-\frac{1}{2}-\varepsilon}
	= \sum \frac{|a_n|}{n^{\frac{1}{k}}}n^{-\frac{1}{2}-\varepsilon+\frac{1}{k}}
	\leq \Big( \sum \frac{|a_n|^2}{n^{\frac{2}{k}}} \Big)^{\frac{1}{2}} \Big(\sum \frac{1}{n^{1+2\varepsilon-\frac{2}{k}}} \Big)^{\frac{1}{2}} \,,
\]
and these two series are convergent. This implies that $D^{(N)} \circ \phi = \sum_{n=1}^{N} a_{n} n^{-\phi}$ converges uniformly on $\mathbb{C}_{\theta'}$ to $D \circ \phi$. Using again \eqref{monomio} we have
\[
	(D^{(N)} \circ \phi)_{\chi}(s) 
	= \sum_{n=1}^{N} a_{n}\chi(n)^{c_{0}}n^{-\phi_{\chi}(s)}
	=\Big( (D^{(N)})_{\chi^{c_{0}}} \circ \phi_{\chi} \Big) (s)
\]
on $\mathbb{C}_{\theta'}$. Since taking vertical limits is continuous with respect to the uniform convergence, this gives
\[
	(D\circ\phi)_{\chi} = D_{\chi^{c_{0}}}\circ\phi_{\chi}
\]
on $\mathbb{C}_{\theta'}$. So, these are two holomorphic functions on $\mathbb{C}_{\theta}$ that coincide on a smaller half-plane. The identity principle yields the claim.
\end{proof}
	
The proof of the following result follows word by word that of \cite[Proposition~5.1]{gordon1999composition}, so we omit it.
		
\begin{prop}\label{prop: extension}
Let $\phi:\C_{\frac{1}{2}}\rightarrow\C_{\frac{1}{2}}$ be a holomorphic function for which the composition operator $C_{\phi}:\mathcal{H}_{+}\rightarrow\mathcal{H}_{+}$ is well defined and continuous. Then almost every (with respect to $\chi$) function  $\phi_{\chi}$ has an analytic extension to $\C_{+}$.
\end{prop}
		
We need one more result before we proceed to the proof of Theorem~\ref{teo: op cont}. We know from \cite[Theorem~4.1]{hedenmalm1995hilbert} that if $D\in\mathcal{H}^{2}$ then for almost every $\chi\in\Xi$ the vertical limit $D_{\chi}$ extends to a holomorphic function in $\C_{+}$. We need the counterpart in $\mathcal{H}_{+}$ of this fact.
		
\begin{prop} \label{pretenders}
Let $D \in \mathcal{H}_{+}$, then $D_{\chi}$ extends to a holomorphic function on $\mathbb{C}_{+}$ for almost every $\chi \in \Xi$
\end{prop}
\begin{proof}
Take $D=\sum a_{n} n^{-s} \in \mathcal{H}_{+}$ and observe that, for each fixed $k$,  since $D_{k} \in \mathcal{H}^{2}$, \cite[Theorem~4.1]{hedenmalm1995hilbert} implies that $(D_{k})_{\chi}$ has a holomorphic extension (let us call it $f_{k,\chi}$) to $\mathbb{C}_{+}$. Define the set
\[
		A_{k} :=\big\{ \chi\in\Xi \colon (D_{k})_{\chi} \mbox{ extends holomorphically to } \C_{+} \Big\}
\]
and $A = \bigcap\limits_{k \in \mathbb{N}} A_{k}$ (note that  $\Xi \setminus A$ has null measure in $\Xi$). We fix now $\chi \in A$ and $k \in \mathbb{N}$ and observe that, for $s \in \mathbb{C}_{\frac{1}{2} + \frac{1}{k}}$ we have
\[
	D_{\chi}(s)=\sum_{n=1}^{\infty} a_{n} \chi(n) \frac{1}{n^{s}}
	= \sum_{n=1}^{\infty} \frac{a_n}{n^{\frac{1}{k}}}\chi(n) \frac{1}{n^{s-\frac{1}{k}}} 
	= (D_{k})_{\chi} \big(s-\tfrac{1}{k} \big)
	= f_{k,\chi}  \big(s-\tfrac{1}{k} \big)\,.
\]
Define now $g_{k,\chi} : \mathbb{C}_{\frac{1}{k}} \to \mathbb{C}$ by $g_{k,\chi} (s) = f_{k,\chi}  \big(s-\tfrac{1}{k} \big)$, that is clearly holomorphic and satisfies $g_{k,\chi} (s) = D_{\chi} (s)$ for every $s \in \mathbb{C}_{\frac{1}{2} + \frac{1}{k} }$. Note now that, if $k<l$ then $g_{k,\chi}$ and $g_{l,\chi}$ coincide on $\mathbb{C}_{\frac{1}{k}}$, and we can define $g_{\chi} : \mathbb{C}_{+} \to \mathbb{C}$ by $g_{\chi}(s) = g_{k,\chi}(s)$ if $s \in \mathbb{C}_{\frac{1}{k}}$. This is a holomorphic function that extends $D_{\chi}$ to $\mathbb{C}_{+}$.
\end{proof}
We can now give the announced alternative proof necessity in Theorem~\ref{teo: op cont}. 
		
\begin{proof}[Alternative proof of necessity in Theorem \ref{teo: op cont}]
The proof follows essentially the same lines as the proof of necessity in \cite[Theorem~B]{gordon1999composition}, so we only sketch it. Let us assume that $\phi:\C_{\frac{1}{2}}\rightarrow\C_{\frac{1}{2}}$ generates a continuous composition operator $C_{\phi}:\mathcal{H}_{+}\rightarrow\mathcal{H}_{+}$. Then, by Proposition~\ref{prop: 3}, for every $D \in \mathcal{H}_{+}$ and every $\chi\in\Xi$ we have
\begin{equation} \label{blackseeds}
			(D\circ\phi)_{\chi} = D_{\chi^{c_{0}}} \circ \phi_{\chi}
\end{equation}
on $\C_{\frac{1}{2}}$. If we can find $\chi \in \Xi$ for which $\phi_{\chi}$ extends to a holomorphic function on $\mathbb{C}_{+}$ satisfying 
\begin{equation} \label{nickcave}
		\phi_{\chi} (\mathbb{C}_{+}) \subseteq \mathbb{C}_{+} \, \text{ if } \, c_{0} \in  \mathbb{N} \, \text{ or } \,
		\phi_{\chi} (\mathbb{C}_{+}) \subseteq \mathbb{C}_{\frac{1}{2}} \text{ if } c_{0} =0 \,,
\end{equation}
then since $\phi = \big(\phi_{\chi}  \big)_{\chi^{-1}}$, \cite[Proposition~4.1]{gordon1999composition} yields our claim.\\
By assumption $D \circ \phi \in \mathcal{H}_{+}$ and, by Proposition~\ref{pretenders}, $(D \circ \phi)_{\chi}$ extends to a holomorphic function on $\mathbb{C}_{+}$ for almost every $\chi$. On the other hand, by Proposition~\ref{prop: extension}, $\phi_{\chi}$ extends to a holomorphic function on $\mathbb{C}_{+}$ for almost all $\chi$. \\
Suppose that  $c_{0} = 1,2, \ldots$. Since the $\chi\mapsto\chi^{c_{0}}$ preserves measure, the function $D_{\chi^{c_{0}}}$ extends holomorphically to $\C_{+}$ for almost every $\chi\in\Xi$. Then there is a set of full measure in $\Xi$ for which $D_{\chi}$, $(D\circ\phi)_{\chi}$ and $\phi_{\chi}$ can be extended analytically to $\mathbb{C}_{+}$ and then \eqref{blackseeds} holds on  $\mathbb{C}_{+}$. Take one such $\chi$ and, following exactly the same argument as in \cite[page~323]{gordon1999composition}, to see that \eqref{nickcave} holds in this case it suffices to find some $D \in \mathcal{H}_{+}$ for which $D_{\chi}$ does not extend analytically to any region larger than $\mathbb{C}_{+}$ for almost all $\chi$.\\
If $c_{0}=0$ then we can only assure that $D_{\chi^{c_{0}}} = D$ defines a holomorphic function on $\mathbb{C}_{\frac{1}{2}}$. The argument in \cite[page~324]{gordon1999composition} shows that 
if we can find a Dirichlet series $D \in \mathcal{H}_{+}$ that cannot be extended to a half-plane larger than $\mathbb{C}_{\frac{1}{2}}$ then \eqref{nickcave} holds in this case. \\
The series $\sum a_{n} n^{-s}$ given by $a_{n} = \frac{1}{\sqrt{n} \log n}$ if $n$ is prime and $0$ otherwise clearly belongs to $\mathcal{H}_{+}$ and in \cite[page~325]{gordon1999composition}
it is shown that satisfies the two required conditions.
\end{proof}

\section{Superposition operators} \label{sec:super}

If $\varphi : \mathbb{C} \to \mathbb{C}$ is some function and $f$ belongs to some space of functions (say $E$) defined on a subset of $\mathbb{C}$ with complex values we may consider $S_{\varphi} (f)= \varphi \circ f$. In this way we define an 
operator $S_{\varphi}$ (called superposition operator) on $E$. We draw now our attention to the existence of superposition operators between spaces of Dirichlet series. This was first considered in \cite{bayart2019composition}, where it is shown
that a function $\varphi$ defines a superposition operator $S_{\varphi} : \mathcal{H}^{\infty} \to \mathcal{H}^{\infty}$ if and only if $\varphi$ is entire, and $S_{\varphi}: \mathcal{H}^{p} \to \mathcal{H}^{q}$ is well defined if and only if 
$\varphi$ is a polynomial  degree less or equal to $\lfloor{\frac{p}{q}}\rfloor$. In this section we carry on with this study, considering operators defined on $\mathcal{H}^{\infty}_{+}$ or $\mathcal{H}_{+}$. We show that 
$S_{\varphi} : \mathcal{H}_{+}^{\infty} \to \mathcal{H}_{+}^{\infty}$ is well defined if and only if $\varphi$ is entire (as in the case of $\mathcal{H}^{\infty}$). The case of $\mathcal{H}_{+}$ is different. We give examples of entire functions $\varphi$ that are not polynomials but that define superposition operators $S_{\varphi} : \mathcal{H}_{+} \to \mathcal{H}_{+}$, but that there are entire functions that do not define superposition operators. 
We begin with a technical lemma.

\begin{lem} \label{itaca}
Let $\varphi : \mathbb{C}\setminus \{0\} \to \mathbb{C}$ be such that $f_{R} (s) = \varphi \big( \frac{R}{2^{s}} \big) \in \mathcal{D}$ for every $R>0$ and $\sup_{R} \sigma_{c}(f_{R}) = \sigma < \infty$.Then $\varphi $ extends to an entire function.
\end{lem}
\begin{proof}
Fix some $R>0$, then the function $s \mapsto \varphi \big( \frac{R}{2^{s}}\big)$ defines a holomorphic function on $\mathbb{C}_{\sigma}$. Then, taking two different branches of the complex logarithm we have that $\varphi$ is holomorphic on $\mathbb{D} (0, R/2^{\sigma}) \setminus \{0\}$. Since $R>0$ was arbitrary, we have that $\varphi$ is holomorphic on $\mathbb{C} \setminus \{0\}$. Now, the Dirichlet series defined on $\mathbb{C}_{\sigma}$ by  $s \mapsto \varphi \big( \frac{1}{2^{s}}\big)$ is bounded on some half plane $\mathbb{C}_{\theta}$ and, then,
\[
	\sup_{|z| \leq \frac{1}{2^{\theta}}} |\varphi(z)| 
	= \sup_{s \in \C_{\theta}} \Big|\varphi \Big( \tfrac{1}{2^{s}} \Big) \Big| <\infty \,.
\]
Hence $0$ is an isolated singularity of a bounded function in $\mathbb{D}(0,\frac{1}{2^{\theta}})$. Hence $\varphi$ can be extended holomorphically to $0$ and this completes the proof.
\end{proof}

\begin{prop}
Let $\varphi : \mathbb{C} \to \mathbb{C}$ be any function. Then,
\begin{enumerate}
		\item \label{app1} $S_{\varphi} : \mathcal{H}_{+}^{\infty} \to \mathcal{H}_{+}^{\infty}$ is well defined if and only if $\varphi$ is entire,
		
		\item \label{app2} if $S_{\varphi} : \mathcal{H}_{+} \to \mathcal{H}_{+}$ is well defined, then $\varphi$ is entire.
\end{enumerate}
\end{prop}

\begin{proof}
First note that, in both cases, if $S_{\varphi}$ is well-defined then $\varphi$ is holomorphic at the origin. Indeed, if $D(s):=-1 + 1^{-s}$ then $\varphi \circ D$ is holomorphic in $\mathbb{C}_{1/2}$ and then $\varphi(s)= \varphi \circ D \circ D^{-1} $ is holomorphic at $s=0$.
The fact that $\varphi$ is entire then follows from Lemma~\ref{itaca}. \\
It is only left to show that if $\varphi$ is entire, then  $S_{\varphi} : \mathcal{H}_{+}^{\infty} \to \mathcal{H}_{+}^{\infty}$ is well defined. Take, then, $D \in \mathcal{H}_{+}^{\infty}$ and fix $k \in \mathbb{N}$. Then
\[
	\sup_{\re(s)>\frac{1}{k}} |\varphi\circ D (s)| 
	= \sup_{\re(s)> \frac{1}{k}} |\varphi(D)(s)|
	=\sup_{\re(s)>0} \Big|\varphi \Big ( D \big(s+ \tfrac{1}{k} \big) \Big) \Big|
	=\sup_{\re(s)>0} |\varphi(D_{k}(s))| \,,
\]
and this supremum is finite because $D_{k} \in \mathcal{H}^{\infty}$ and the superposition operator is well defined from $\mathcal{H}^{\infty}$ to $\mathcal{H}^{\infty}$ (because $\varphi$ is entire). 
Since this 	holds fore very $k$ we conclude that $\varphi\circ D \in \mathcal{H}_{+}^{\infty}$.
\end{proof}

So the behaviour of superposition operators on $\mathcal{H}^{\infty}$ and on $\mathcal{H}^{\infty}_{+}$ is essentially the same. This is not the case when we look at superposition operators defined on 
$\mathcal{H}_{+}$. In this case, every polynomial (of any degree) defines a superposition operator. The reason for this is that (unlike for $\mathcal{H}^{p}$), the space $\mathcal{H}_{+}$ is an algebra. 

\begin{prop} \label{Fr alg}
The space $\mathcal{H}_{+}$ is a Fr\'echet algebra.
\end{prop}
\begin{proof}
Fix $m$ and take two Dirichlet polynomials
	$P$ and $Q$. Then $PQ$ is again a Dirichlet polynomial and (recall \eqref{seminormas})
\begin{align*}
	{}\Vert PQ\Vert_{2,m}&\leq C_{m} \Vert PQ \Vert_{1,2m}
	= C_{m} \Vert (PQ)_{2m} \Vert_{1}
	= C_{m} \lim_{T \to \infty} \frac{1}{2T}\int_{-T}^{T} |P_{2m}(it) Q_{2m}(it)| \mathrm{d}t\\
	& \leq C_{m} \lim_{T \to \infty} \frac{1}{2T} \Bigg( \int_{-T}^{-T} |P_{2m}(it)|^2 \mathrm{d}t \Bigg)^{\frac{1}{2}} 
	\Bigg( \int^{T}_{-T} |Q_{2m}(it)|^2 \mathrm{d}t \Bigg)^{\frac{1}{2}}\\
	& = C_{m} \lim_{T \to \infty} \Bigg( \frac{1}{2T} \int^{T}_{-T} |P_{2m}(it)|^2 \mathrm{d}t \Bigg)^{\frac{1}{2}}
	\lim_{R \to \infty} \Bigg( \frac{1}{2R} \int^{R}_{-R} |Q_{2m}(ir)|^2 \mathrm{d}r \Bigg)^{\frac{1}{2}}\\
	& = C_{m} \Vert P_{2m} \Vert_2\Vert Q_{2m}\Vert_2
	= C_{m} \Vert P \Vert_{2,2m} \Vert Q \Vert_{2,2m} \,.
\end{align*}
Take now two Dirichlet series $D_{1}, D_{2} \in \mathcal{H}_{+}$ and choose sequences of Dirichlet polynomials $(P_{j})_{j}$ and $(Q_{j})_{j}$ converging to $D_{1}$ and $D_{2}$, respectively. Our first step is to show that $(P_{j}Q_{j})_{j}$ is a Cauchy sequence. Note that for each $k$ we may find $M_{k}$ so that $\Vert P_{j} \Vert_{2,k} \leq  M_{k}$ and $\Vert Q_{j} \Vert_{2,k} \leq  M_{k}$. With this at hand we immediately have, for each $m$
\begin{multline*}
	\Vert P_{j}Q_{j} - P_{i} Q_{i} \Vert_{2,m} 
	\leq C_{m} ( \Vert P_{j} \Vert_{2,2m} \Vert Q_{j} - Q_{i} \Vert_{2,2m} + \Vert P_{i} - P_{j} \Vert_{2,2m} \Vert Q_{i} \Vert_{2,2m} ) \\
	\leq C_{m} M_{2m} (\Vert Q_{j} - Q_{i} \Vert_{2,2m} + \Vert P_{i} - P_{j}\Vert_{2,2m} ) \,.
\end{multline*}
Hence $(P_{j}Q_{j})_{j}$ is a Cauchy sequence and then converges to some $D \in \mathcal{H}_{+}$. Given $s \in \mathbb{C}_{\frac{1}{2}}$ we have
\[
	D(s) = \lim_{j} (P_{j}Q_{j}) (s) = \lim_{j} P_{j} (s) \lim_{j} Q_{j} (s) = D_{1}(s) D_{2}(s) \,,
\]
and this shows that $D_{1}D_{2} \in \mathcal{H}_{+}$, and a standard argument shows that
\[
	\Vert D_{1} D_{2} \Vert_{2,m} \leq C_{m}M_{2m} \Vert D_{1} \Vert_{2,2m} \Vert D_{2} \Vert_{2,2m} \,.
\]
\end{proof}
\begin{cor}
If $\varphi :\mathbb{C} \to \mathbb{C}$ is a polynomial, then the superposition operator $S_{\varphi} : \mathcal{H}_{+} \to \mathcal{H}_{+}$ is well defined.
\end{cor}

We address now the question of whether or not there are entire functions other than polynomials that define a superposition operator on  $\mathcal{H}_{+}$.

\begin{teo} \label{marianne}
There are entire functions that are not polynomials that define a superposition operator on $\mathcal{H}_{+}$, but not every entire function does so.
\end{teo}

There are two questions here to be answered. We deal with each one of them separately. The key point for the first question (on the existence of entire functions defining superposition operators that are not polynomials) is to have a good control of the seminorms of the powers of a given Dirichlet series. Let us recall that if $\pi (x) = \sum_{\substack{\mathfrak{p} \leq x \\ \mathfrak{p} \text{ prime}}} 1$ then, by the prime number theorem 
\begin{equation} \label{pi}
\lim_{x \to \infty} \frac{\pi (x)}{\big(\frac{x}{\log x}\big)} =1 \,.
\end{equation}

\begin{lem} \label{comparacion potencias}
Given $m \in \mathbb{N}$ there exist $C_{m}, b_{m} >1$ so that 
\[
	\Vert D^{k} \Vert_{2,m} \leq C_{m} e^{b_{m} k^{2m+1}} \Vert D \Vert_{2,4m}^{k}
\]
for every $D \in \mathcal{H}_{+}$ and every $k \in \mathbb{N}$.
\end{lem}
\begin{proof}
Let $m,k \in \mathbb{N}$ and choose $j_{k,m}$ to be the smallest natural number so that 	$\mathfrak{p}_{j_{k,m}}^{\frac{-1}{4m}}\leq\sqrt{\frac{2}{k}}$ (recall that $\mathfrak{p}_{j_{k,m}}$ is the $j_{k,m}$-th prime number). Then, for every Dirichlet polynomial $P$ we have (recall the proof of Proposition~\ref{obs: igualdal de espacios})
\begin{align*}
	\Vert P^k\Vert_{2,m} & \leq C_{m}\Vert P^k \Vert_{1,2m} 
	= C_{m} \lim_{T \to \infty} \frac{1}{2T} \int_{-T}^{T} |P_{2m}(it)|^k \mathrm{d}t \\
	& = C_{m} \Bigg( \lim_{T \to \infty} \frac{1}{\sqrt[k]{2T}} \Bigg( \int_{-T}^{T} |P_{2m}(it)|^{k} \mathrm{d} t \Bigg)^{\frac{1}{k}}\Bigg)^k
	= C_{m} \Vert P_{2m} \Vert_{k}^{k} 
	= C_{m} \Vert P \Vert_{k,2m}^{k}\\
	& \leq C_{m} \Bigg(\prod_{j=1}^{j_{k,m}} \frac{1}{1-\mathfrak{p}_{j}^{\frac{-1}{4m}}} \Bigg)^{k} \Vert P \Vert_{2,4m}^{k} \,.
\end{align*}
Take now $D \in \mathcal{H}_{+}$ and choose a sequence of Dirichlet polynomials $(P_{i})_{i}$ converging to $D$. Then, for each fixed $k$, the sequence $(P_{i}^{k})_{i}$ converges to $D^{k}$ (check again the proof of Proposition~\ref{Fr alg}) and, then, 
\[
	\Vert D^{k}\Vert_{2,m}
	= \lim_{i} \Vert P_{i}^k\Vert_{2,m}
	\leq \lim_{i} C_{m} \Bigg( \prod_{j=1}^{j_{k,m}} \frac{1}{1-\mathfrak{p}_{j}^{\frac{-1}{4m}}} \Bigg)^{k} \Vert P_i \Vert_{2,4m}^{k}
	= C_{m} \Bigg( \prod_{j=1}^{j_{k,m}} \frac{1}{1-\mathfrak{p}_{j}^{\frac{-1}{4m}}} \Bigg)^{k} \Vert D \Vert_{2,4m}^{k}
\]
for every $m$. As a straightforward consequence of \eqref{pi} we can find $A_{m} >1$ so that
$\pi\Big(\big(\frac{k}{2}\big)^{2m} \Big) \leq A_{m} k^{2m}$ for every $k$. 
Taking $c_{m} =-\log(2^{\frac{1}{4m}}-1)+\frac{1}{4m}\log(2)$, we have $\frac{1}{1-x} \leq e^{x+c_m}$ for every  $x \in \Big[ 0,\frac{1}{2^{\frac{1}{4m}}} \Big]$. Then, taking $b_{m} = c_{m} + A_{m}$ we get
\[
	\prod_{j=1}^{j_{k,m}} \frac{1}{1-\mathfrak{p}_{j}^{\frac{-1}{4m}}} 
	\leq \prod_{j=1}^{j_{k,m}} e^{\mathfrak{p}_{j}^{\frac{-1}{4m}}+c_m}
	= e^{c_{m} j_{k,m}} e^{\sum\limits_{j=1}^{j_{k,m}} \frac{1}{\mathfrak{p}_{j}^{\frac{1}{4m}}}} 
	\leq e^{c_{m} j_{k,m}} e^{A_{m}\frac{j_{k,m}^{1-\frac{1}{4m}}}{ \log(j_{k,m})}}
	\leq e^{b_{m}k^{2m}}e^{A_{m}k^{2m}}
	=e^{k^{2m} b_{m}}\,. \qedhere
\]
\end{proof}

Now we can already answer our first question in Theorem~\ref{marianne}, giving an example of an entire function that is not a polynomial and that produces a well defined superposition operator on $\mathcal{H}_{+}$.

\begin{ej} \label{exam: funcion superposicion}
The entire function given by $\varphi(z)=\sum\limits_{k=0}^{\infty} \frac{1}{e^{k^k}}z^k$ defines an superposition operator on $\mathcal{H}_{+}$. Fix $m\in\mathbb{N}$ and $D\in\mathcal{H}_{+}$. Then we have
\[
	\Big\Vert \sum_{k=N}^{M} \frac{1}{e^{k^k}}D^{k} \Big\Vert_{2,m}
	\leq \sum_{k=N}^{M} \frac{1}{e^{k^k}} \Vert D^{k} \Vert_{2,m} 
	\leq C_{m} \sum_{k=N}^{M} \frac{\Big( e^{b_{m}k^{2m}} \Vert D\Vert_{2,4m} \Big)^k}{e^{k^k}}
	= C_{m} \sum_{k=N}^{M} \bigg( \frac{e^{b_{m}k^{2m}} \Vert D \Vert_{2,4m}}{e^{k^{k-1}}} \bigg)^{k}
\]
for every $M>N$. Since $\frac{e^{\widetilde{b}_{m}k^{2m}}\Vert D\Vert_{2,4m}}{{e^{k^{k-1}}}}<1$ for big enough $k$, the latter term tends to $0$ as $M$ and $N$ go to $\infty$. Then $\sum\limits_{k=0}^{N}{\frac{1}{e^{k^{k}}}D^{k}}$ is a Cauchy sequence in $\mathcal{H}_{+}$ and therefore 
converges to some $\widetilde{D} \in \mathcal{H}_{+}$. In particular this implies $\sum\limits_{k=0}^{N}{\frac{1}{e^{k^{k}}}D^{k}(s)}\rightarrow\widetilde{D}(s)$ for every $s\in\mathbb{C}_{\frac{1}{2}}$. On the other hand, if $s\in\mathbb{C}_{\frac{1}{2}}$, then $\sum\limits_{k=0}^{N} \frac{1}{e^{k^{k}}}D^{k}(s) \rightarrow \sum\limits_{k=0}^{\infty} \frac{1}{e^{k^{k}}}D^{k}(s)=\varphi(D)(s)$.
This shows that $\varphi(D)\in\mathcal{H}_{+}$ and $S_{\varphi}$ is a well defined superposition operator on $\mathcal{H}_{+}$.
\end{ej}

This example settles the first part of Theorem~\ref{marianne}. In order to address the second part (if every entire function  $\varphi$ defines a superposition operator $S_{\varphi}:\mathcal{H}_{+} \to \mathcal{H}_{+}$)
we change slightly our perspective. 
Let $H(\C)$ be  the space of entire functions endowed with the topology of uniform convergence on compact sets.
Given a Dirichlet series $D \in \mathcal{H}_{+}$, we say that $D$ is of composition  in $\mathcal{H}_{+}$ if  $\varphi \circ D$ remains in $\mathcal{H}_{+}$, for every entire function $\varphi$ (in other words, the operator $C_D: H(\C) \to \mathcal{H}_{+}$ given by $\varphi \mapsto \varphi \circ D$ is well defined). 
It is plain that every entire function defines a superposition operator if and only if every Dirichlet series in $\mathcal{H}_{+}$ is of composition. So, our goal in order to answer (in the negative) the second question in Theorem~\ref{marianne} is to find a Dirichlet series in $\mathcal{H}_{+}$ that is not of composition.\\
Let us suppose that $D$ is such that the operator $C_D$ is well defined and let us see that it is also continuous. Take a sequence 
$(\varphi_{m})\subseteq H(\mathbb{C})$ converging  to some $\varphi\in H(\mathbb{C})$ and assume that $C_D(\varphi_m)= \varphi_{m}\circ{D}$ converges (in $\mathcal{H}_{+}$) to $\widetilde{D}$. On the one hand,  since $\varphi_{m}\to{\varphi}$ in $H(\mathbb{C})$ then $\varphi_{m}(D(s))\to{\varphi(D(s))}$ for every $s \in \mathbb{C}_{\frac{1}{2}}$. On the other hand $\varphi_{m}\circ{D}\to\widetilde{D}$ in $\mathcal{H}_{+}$ so in particular $\varphi_{m}\circ{D}(s)\to\widetilde{D}(s)$ for all $s\in\mathbb{C}_{\frac{1}{2}}$. Therefore $\varphi\circ{D}$ coincides with $\widetilde{D}$ in the half-plane $\mathbb{C}_{\frac{1}{2}}$, so they must be the same Dirichlet series. The closed-graph theorem gives that not only is $C_{D}$ well defined, but also continuous.\\

In order to find the series that is not of composition we provide now a necessary and sufficient condition for a Dirichlet series to be so.
\begin{prop} \label{andueza}
Let $D \in \mathcal{H}_{+}$, then $D$ is of composition in $\mathcal{H}_{+}$ if and only if for every $m \in \N$ there is a constant $C>0$ such that 
\begin{equation} \label{murdoch}
	\Vert D^k \Vert_{2,m}^{\frac{1}{k}} \leq C \,,
\end{equation}
for every $k \in \N$.
\end{prop}
\begin{proof}
Suppose $D$ is of composition in $\mathcal{H}_{+}$ then the operator $C_{D}: H(\mathbb{C}) \to \mathcal H_{+}$ is continuous. 
Given $m \in \mathbb N$, there is a constant $A>0$ and $j\in \N$ so that $\Vert C_{D}(\varphi)\Vert_{2,m}\leq A \sup_{ \vert z \vert \leq j}\vert \varphi(z) \vert$. Taking $\varphi(z)=z^k$ we have
\[
	\Vert D^{k} \Vert_{2,m} \leq A \sup_{|z|\leq j} |z|^{k} \leq Aj^{k}  \leq C^k \,.
\]
for $C>0$ sufficiently large.\\	
For the converse if $\varphi (z) = \sum_{k=0}^\infty a_k z^k$, note that if \eqref{murdoch} holds then 
\[
	\sum_{k=0}^\infty \vert a_k \vert \Vert D^k \Vert_{2,m} \leq \sum_{k=0}^\infty \vert a_k \vert C^k < \infty
\]
for every $m \in \N$. Then the series converges and the operator given by $C_{D}(\varphi)=\sum_{k=0}^\infty a_k D^k$ is well defined.
\end{proof}

With this we can give an example of a Dirichlet series in $\mathcal{H}_{+}$ that is not of composition, answering the second question in Theorem~\ref{marianne}. Let us recall that he first Chebishev function $\vartheta (x)= \sum_{\substack{\mathfrak{p} \leq x \\ \mathfrak{p} \text{ prime}}} \log(\mathfrak{p})$ satisfies
\begin{equation} \label{chebishev}
\lim_{x \to \infty} \frac{\vartheta (x)}{x} =1 \,.
\end{equation}
This is equivalent to the prime number theorem, and an accurate estimate can be found at \cite{rosserschoenfeld1962}).

\begin{ej} \label{ejemplo}
Let us consider $D=\sum\limits_{n=1}^{\infty}{\frac{1}{\sqrt{n}}n^{-s}}$, that obviously belongs to $\mathcal{H}_{+}$, and let us see that it is not a composition Dirichlet series. Observe first that $D(s) = \zeta (s+ \frac{1}{2})$ for every $s \in \mathbb{C}_{\frac{1}{2}}$ (being $\zeta$ the Riemann's zeta function) and that, then, 
\[
	D^{k}=\sum_{n=1}^{\infty}\frac{d_{k}(n)}{\sqrt{n}}n^{-s} \,,
\]
where the $d_{k}(n)$s are the coefficients of $\zeta^{k}$, that is $d_{k}(n)=\sum_{n_{1}\cdots{n_{k}}=n}1$.\\
Fix $m \in \mathbb{N}$, define $\sigma = \frac{1}{2m}$ and consider $0 < \delta < 1$ such that $\omega=\frac{2(1-\delta)}{1+\delta}-(1+\sigma)(1+\delta)>0$ (note that such a $\delta$ exists because the previous expression is positive for $\delta =0$). \\
By \eqref{pi} and \eqref{chebishev} we can choose $x_{0}$ so that
\begin{equation} \label{dolors}
(1- \delta ) \frac{x}{\log x} \leq \pi (x) \leq (1+ \delta ) \frac{x}{\log x}  \, \text{ and } \,
(1 - \delta) x \leq \vartheta (x) \leq (1+ \delta) x
\end{equation}
for every $x \geq x_{0}$. Pick now $k_{0} \in \mathbb{N}$ so that $k_{0} \geq x_{0}$ and, for each $k \geq k_{0}$ define $x_{k}=k^{1+\delta}$ and $n_{k} = \prod_{\mathfrak{p} \leq x_{k}} \mathfrak{p} = e^{\vartheta (x_{k})}$. Observe that in this case $d_{k}(n_{k})=k^{\pi(x_{k})}$. Then, taking \eqref{dolors} into account we get
\begin{multline*}
	\bigg( \frac{d_{k}^{2}(n_{k})}{n_{k}^{1+\frac{1}{2m}}} \bigg)^{\frac{1}{k}}
	=\bigg( \frac{k^{2\pi(x_{k})}}{e^{\vartheta(x_{k})(1+\frac{1}{2m})}} \bigg)^{\frac{1}{k}}
	= \exp\bigg(\tfrac{1}{k} \Big( 2 \pi(x_{k}) \log(k) - ( 1+\sigma ) \vartheta(x_{k}) \Big) \bigg) \\
	>
	\exp \bigg( \tfrac{k^{1+\delta}}{k} \Big( \tfrac{2(1-\delta)}{1+\delta}-(1+\sigma)(1+\delta) \Big) \bigg)\
	=e^{k^{\delta}\omega} \,.
\end{multline*}
Hence
\[
	\Vert D^{k} \Vert_{2,4m}^{\frac{1}{k}}
	= \Bigg( \sum_{n=1}^{\infty} \frac{d_{k}^{2}(n)}{n^{1+\frac{1}{2m}}} \Bigg)^{\frac{1}{2k}}
	\geq \Bigg( \frac{d_{k}^{2}(n_{k})}{n_{k}^{1+\frac{1}{2m}}} \Bigg)^{\frac{1}{2k}} 
	> e^{k^{\delta}\frac{\omega}{2}} \,.
\]
Proposition~\ref{andueza} gives that $D$ cannot be a composition Dirichlet series in $\mathcal{H}_{+}$.
\end{ej}

Our last result shows that, in some sense, if we want $\varphi$ to define a superposition operator, then its coefficients have to go to $0$ quite fast (recall also Example~\ref{exam: funcion superposicion}).\\
Let us note that if $S_{\varphi}$ is of superposition then $S_{\varphi}(\zeta(s+\frac{1}{2})) \in \mathcal{H}_{+}$, in particular $\varphi\left(\zeta(s+\frac{1}{2}+\varepsilon)\right) \in \mathcal{H}^{2}$ for all $\varepsilon>0$. Rearranging terms we have 
\[
\sum_{n=1}^{\infty} \Bigg( \sum_{k=0}^{\infty} a_{k} d_{k}(n) \Bigg) \frac{1}{n^{\frac{1}{2}+\varepsilon}}n^{-s} \in \mathcal{H}^{2}\,.
\]
What we are going to do now is to show that there are functions for which this does not hold.

\begin{prop}
For every $0 <C< 2 $ there exists some $\varepsilon >0$ so that the series
\begin{equation*}
 \sum_{n=1}^{\infty} \Bigg( \sum_{k=0}^{\infty} \frac{ d_k(n) }{e^{k^C}} \Bigg)
 \frac{1}{n^{\frac{1}{2}+\varepsilon}}
 n^{-s}
\end{equation*}{}
does not belong to $\mathcal{H}^{2}$. In particular the function $\varphi (z) = \sum_{k=0}^{\infty} \frac{1}{e^{k^{C}}} z^{k}$ does not define a superposition operator.
\end{prop}

\begin{proof}
Considering as before $ n:= \prod_{\mathfrak{p} \leq x} \mathfrak{p} $ with large enough $x$
we get
\begin{equation}\label{e:lowboundexpkC}
 \left( \sum_{k=0}^{\infty} \frac{ d_k(n) }{e^{k^C}} \right)
 \frac{1}{n^{\frac{1}{2}+\eps}} \geq 
 \frac{ d_k(n) }{e^{k^C}} 
 \frac{1}{n^{\frac{1}{2}+\eps}} 
  = 
  \exp\left(
  \log(k)\pi(x) - k^C - \vartheta(x)\left(\frac{1}{2} + \eps\right)
  \right).
\end{equation}
For a given $\delta >0 $, this last exponent is bounded below by
\[
    x  \frac{\log k}{\log x}(1-\delta) - k^C
    -x\left(\frac{1}{2} + \eps\right)(1+\delta)   
\]
as long as \eqref{dolors} holds.

Let us choose $C'$ such that $C < C '< 2, $ $x_k := k^{C'}$ and $\delta,\eps >0$
small enough to that
\begin{align*}
      \omega := \frac{1-\delta}{C '} - \left(\frac{1}{2} + \eps\right)(1+\delta)    >  0.
\end{align*}
Then \eqref{e:lowboundexpkC} becomes
\begin{align}
 \left( \sum_{k=0}^{\infty} \frac{ d_k(n) }{e^{k^C}} \right)
 \frac{1}{n^{\frac{1}{2}+\eps}} \geq &
  \exp\left(
  \omega k^{C'} - k^C 
  \right)
\end{align}
which tends to $+\infty$ with $k$.
\end{proof}

The same argument shows that if $\varphi (z) = \sum_{k=0}^{\infty} a_{k} z^{k}$ is such that  there exist $m>0$ with $a_{k} e^{k^{C}}>m$ for large enough $k$ and $1<C<2$, then  it cannot define a superposition operator.
In particular, the exponential function
$\exp(z) = \sum z^k /k!$ does not define a superposition operator,
since $\log (k!) = k \log k + O(k)$.

\section{Differentiation and integration operators}  \label{sect:dif integ}

We finish this note by looking at the classical differentiation operator (that brings a holomorphic function to its derivative) and at its inverse, the integration operator. These two operators, defined on the space $\mathcal{H}_{+}^{\infty}$ have been studied in \cite{bonetoperator}. The situation in $\mathcal{H}_{+}$ is quite close and the arguments are rather similar, so we  sketch them, only pointing out the steps on which they are different. \\

As we have already explained a Dirichlet series $D=\sum a_{n} n^{-s}$ defines a holomorphic function on $\mathbb{C}_{\sigma_{c}(D)}$. Then its derivative is again a Dirichlet series obtained simply by differentiating term by term, that is 
\[
D'(s)=
-\sum_{n=2}^{\infty} a_{n}\log(n) n^{-s} \,,
\]
and has the same abscissa of convergence as $D$ (see e.g. \cite[Theorem~11.12]{apostol1976}). We can then consider the differentiation operator $\mathbf{D} : \mathcal{D} \to \mathcal{D}$ defined as $\mathbf{D}(D)= D'$.

\begin{prop} \label{petenera}
The differentiation operator $\mathbf{D} : \mathcal{H}_{+} \to \mathcal{H}_{+}$ is continuous and satisfies
\[
\mathbf{D}(\mathcal{H}_{+}) = \mathcal{H}_{+,0} := \Big\{\sum{a_{n} n^{-s}}\in\mathcal{H}_{+} \colon  a_{1}=0 \Big\} \,.
\]
\end{prop}
\begin{proof}
Let $D=\sum a_{n}n^{-s} \in \mathcal{H}_{+}$, fix $k\in\mathbb{N}$ and set $C = \sup_{n} \frac{\log n}{n^{1/(2k)}}$. Then
\[
\Vert D' \Vert_{2,k}
 = \bigg( \sum_{n=2}^{\infty} \frac{ \vert a_{n} \vert^{2}}{n^{1/k}} \frac{\log(n)^{2}}{n^{1/k}} \bigg)^{\frac{1}{2}}
\leq C \bigg( \sum\limits_{n=2}^{\infty} \frac{\vert a_{n} \vert^{2}}{n^{1/k}} \bigg)^{\frac{1}{2}}
\leq C \bigg( \sum\limits_{n=1}^{\infty} \frac{\vert a_{n} \vert^{2}}{n^{2/(2k)}} \bigg)^{\frac{1}{2}}
= C \Vert D \Vert_{2,2k} \,,
\]
and $\mathbf{D}$ is well defined and continuous.\\
Clearly $\mathbf{D}(\mathcal{H}_{+}) \subseteq \mathcal{H}_{+,0}$.  On the other hand if $D=\sum_{n=2}^{\infty} a_{n}n^{-s} \in \mathcal{H}_{+,0}$, then is plain that $\widetilde{D}=\sum_{n=2}^{\infty} \frac{-a_{n}}{\log(n)}n^{-s} \in \mathcal{H}_{+}$ and $\mathbf{D}(\widetilde{D})=D$, hence $\mathcal{D}(\mathcal{H}_{+})=\mathcal{H}_{+,0}$.
\end{proof}

A simple computation shows that, for each fixed $N$, the coefficient operator $\sum a_{n} n^{-s} \mapsto a_{N}$ is continuous on $\mathcal{H}_{+}$. Then the space  $\mathcal{H}_{+,0}$ that we have just defined is closed. \\

We are also interested in the inverse operator $\mathbf{J}$ defined for Dirichlet series $\sum a_{n}n^{-s}$ for which $a_{1}=0$ as follows
\[
\mathbf{J} \Big(\sum_{n=2}^{\infty} a_{n}n^{-s} \Big) = - \sum_{n=2}^{\infty} \frac{a_{n}}{\log(n)}n^{-s} \,.
\]
Considered as an operator $\mathcal{H}_{+,0} \to \mathcal{H}_{+,0}$, it is clearly well defined and continuous, since
\[
\Vert \mathbf{J}(D) \Vert_{2,k} 
= \bigg( \sum_{n=2}^{\infty} \frac{\vert a_{n} \vert^{2}}{n^{2/k}\log(n)^{2}} \bigg)^{\frac{1}{2}}
\leq \bigg( \sum_{n=2}^{\infty} \frac{\vert a_{n} \vert^{2}}{n^{2/k}} \bigg)^{\frac{1}{2}} = \Vert D \Vert_{2,k}
\]
for every $D=\sum_{n=2}^{\infty} a_{n} n^{-s} \in \mathcal{H}_{+,0}$. A straightforward computation shows that 
$\mathbf{D}\mathbf{J}(D)=D=\mathbf{J}\mathbf{D}(D)$ for all $D\in \mathcal{H}_{+,0}$. Exactly the same argument as in \cite[Theorem~2.3(iii)]{bonetoperator} shows that neither $\mathbf{D}$, nor $\mathbf{J}$ are compact operators.\\

Suppose that $D$ and $E$ are two Dirichlet series with $\sigma_{a}(D), \sigma_{a}(E) < \infty$. By \cite[Theorems~11.12 and~11.10]{apostol1976} $D'$ has also finite abscissa of absolute convergence and the product $D'E = \sum c_{n}n^{-s}$ again converges absolutely at some half-plane. Note also that $c_{1}=0$, and then we may consider $\mathbf{J}(D'E)$. In this way, fixing $D$ we define a Volterra-type operator $\mathbf{V}_{D} : \mathcal{D} \to \mathcal{D}$ given by $\mathbf{V}_{D} (E) = \mathbf{J}(D'E)$. The action of such operators on Hardy spaces was thoroughly studied in \cite{brevigperfektseip2019}, where deep results were given. Later, in \cite[Corollary~2.4]{bonetoperator}, it was shown that the situation in $\mathcal{H}_{+}^{\infty}$ is much easier to handle. Exactly the same arguments as there show that this is also the case in $\mathcal{H}_{+}$, and $\mathbf{V}_{D} : \mathcal{H}_{+} \to \mathcal{H}_{+}$ is well defined (and continuous) if and only if $D \in \mathcal{H}_{+}$.\\

We finish this note by looking at the spectrum of the differentiation and integration operators, in the same spirit as \cite[Theorem~2.6]{bonetoperator}. Let us recall that the resolvent of a linear operator $T:X \to X$ (where $X$ is some Fr\'echet space) is defined as the set $\rho(T,X)$ consisting of those $\lambda \in \mathbb{C}$ for which $(\lambda I - T)$ is bijective and its inverse is continuous. Then the spectrum of $T$ is $\sigma (T,X) = \mathbb{C} \setminus \rho(T,X)$. 

\begin{prop}
We have the following characterization of the spectrums:
\begin{enumerate}
	\item $\sigma (\mathbf{D}, \mathcal{H}_{+,0}) = \{ - \log n \colon n \in \mathbb{N}, \, n \geq 2  \}$.
	\item $\sigma (\mathbf{D}, \mathcal{H}_{+}) = \{0\} \cup \{ - \log n \colon n \in \mathbb{N}, \, n \geq 2  \}$.
	\item $\sigma (\mathbf{J}, \mathcal{H}_{+,0}) = \{ - \frac{1}{\log n} \colon n \in \mathbb{N}, \, n \geq 2  \}$.
\end{enumerate}
\end{prop}
\begin{proof}
Take some $0 \neq \lambda \in \mathbb{C}$ so that $\lambda \neq - \log n$ for every natural number $n \geq 2$ and let us see that $\lambda \in \rho (\mathbf{D}, \mathcal{H}_{+,0})$. Note first that, for $\sum_{n \geq 2} a_{n} n^{-s} \in  \mathcal{H}_{+,0}$ we have
\[
(\lambda I - \mathbf{D}) \Big( \textstyle \sum_{n \geq 2} a_{n} n^{-s} \Big) = 
\displaystyle \sum_{n\geq 2} (\lambda + \log n)a_{n} n^{-s} \,.
\]
Choosing $\mu >0$ so that $\vert \lambda \vert >\mu$ and $\vert \log n + \lambda \vert > \mu$ for every  natural $n \geq 2$ we have
\begin{equation} \label{osresentidos}
\big\Vert \textstyle \sum_{n\geq 2} \frac{a_{n}}{\lambda + \log n} n^{-s} \big\Vert_{2,k}
= \displaystyle \Big( \sum_{n=2}^{\infty} \frac{\vert a_{n} \vert^{2}}{\vert \log n + \lambda \vert^{2} } \frac{1}{n^{2/k}} \Big)^{\frac{1}{2}}
< \frac{1}{\mu} \Big( \sum_{n=2}^{\infty} \frac{\vert a_{n} \vert^{2}}{n^{2/k}} \Big)^{\frac{1}{2}}
= \frac{1}{\mu} \big\Vert \textstyle \sum_{n\geq 2} a_{n} n^{-s} \big\Vert_{2,k} \,
\end{equation}
for every $k$. This shows that $\sum_{n\geq 2} \frac{a_{n}}{\lambda + \log n} n^{-s} \in \mathcal{H}_{+,0}$ and, then $(\lambda I - \mathbf{D}) : \mathcal{H}_{+,0} \to \mathcal{H}_{+,0}$ is surjective. Also \eqref{osresentidos} shows that the inverse $(\lambda I - \mathbf{D})^{-1}$ is continuous, giving our claim. The rest of the proof follows exactly as that of \cite[Theorem~2.6]{bonetoperator}.
\end{proof}

\paragraph*{Acknowledgements.} We would like to warmly thank Jos\'e Bonet, Andreas Defant and Manuel Maestre for enlightening remarks and comments and fruitful discussions that improved the paper. We would also like to thank the referees for their careful reading and helpful comments.\\

The research of T.~Fern\'andez Vidal was supported by PICT 2015-2299.

The research of D.~Galicer was partially supported by CONICET-PIP 11220130100329CO and 2018-04250. 

The research of M.~Mereb was partially supported by CONICET-PIP 11220130100073CO and PICT 2018-03511.

The research of P.~Sevilla-Peris was supported by MICINN and FEDER Project MTM2017-83262-C2-1-P and MECD grant PRX17/00040.

\noindent 
T.~Fern\'andez Vidal, D.~Galicer, M.~Mereb\\
Departamento de Matem\'{a}tica,
Facultad de Cs. Exactas y Naturales, Universidad de Buenos Aires and IMAS-CONICET. Ciudad Universitaria, Pabell\'on I (C1428EGA) C.A.B.A., Argentina, tfernandezvidal@yahoo.com.ar, dgalicer@dm.uba.ar, mmereb@gmail.com\\ 

\noindent P.~Sevilla-Peris\\	
Insitut Universitari de Matem\`atica Pura i Aplicada. Universitat Polit\`ecnica de Val\`encia. Cmno Vera s/n 46022, Spain, psevilla@mat.upv.es
\end{document}